\documentclass[11pt]{amsart}
\usepackage[utf8]{inputenc}
\usepackage{amssymb}
\usepackage{amsmath}
\usepackage{epsfig}
\usepackage{color}
\usepackage{graphics}
\usepackage[normalem]{ulem}
\usepackage{hyperref}

\newcommand{\NN}{{\mathbb N}}
\newcommand{\ZZ}{\mathbb Z}

\def\C{{\mathcal C}}

\def\diam{{\rm diam}}

\def\Aut{{\rm Aut}}

\def\Conj{{\rm Conj}}
\def\Stab{{\rm Stab}}
\numberwithin{equation}{section}

\newtheorem{theo}{Theorem}[section]
\newtheorem{prop}[theo]{Proposition}
\newtheorem{example}[theo]{Example}
\newtheorem{coro}[theo]{Corollary}
\newtheorem{lemma}[theo]{Lemma}
\newtheorem{defi}[theo]{Definition}
\newtheorem{remark}[theo]{Remark}
\RequirePackage[normalem]{ulem} 
\RequirePackage{color}\definecolor{RED}{rgb}{1,0,0}\definecolor{BLUE}{rgb}{0,0,1} 
\title{Group actions with almost normal stabilizers}

\author[]{Mar\'{\i}a Isabel Cortez}
\address{Mar\'{\i}a Isabel Cortez, Facultad de Matem\'aticas, Pontificia Universidad Cat\'olica de Chile. Edificio Rolando Chuaqui, Campus San Joaqu\'{\i}n. Avda. Vicuña Mackenna 4860, Macul, Chile}
\email{maria.cortez@uc.cl}

\author[]{Maik Gr\"oger}
\address{Maik Gr\"oger, Faculty of Mathematics and Computer Science, Jagiellonian University, ul. Prof. S. Łojasiewicza 6, 30-348 Kraków, Poland}
\email{maik.groeger@im.uj.edu.pl}

\author[]{Olga Lukina}
\address{Olga Lukina, Mathematical Institute, Leiden University, P.O. Box 9512,
2300 RA Leiden,
The Netherlands}
\email{o.lukina@math.leidenuniv.nl }

\date{April 1, 2025. Revision April 22, 2026.}



\begin{document}

\begin{abstract}
    In this paper, we consider minimal group actions of countable groups on compact Hausdorff spaces by homeomorphisms. We show that the existence of a point with finite stabilizer imposes strong restrictions on the dynamics: a residual set of points then has stabilizers conjugate to the same almost normal subgroup of the acting group. Under certain conditions on the acting group such an action also has no essential holonomy.
    If the acting group is residually finite, we show that every group action with finite almost normal stabilizers with no essential holonomy arises as an almost finite-to-one factor of an essentially free action.
\end{abstract}

\maketitle

\section{Introduction}

In this note, we study topological and measure-preserving group actions for which most  stabilizers are isomorphic to the same almost normal subgroup of the acting group. Such a property of having `few distinct stabilizers' can be considered as the first generalization of the concept of a topologically free action, as we explain below. We are interested in topological and ergodic properties of such actions. 

By a \emph{topological dynamical system}, or just a \emph{dynamical system}, or a \emph{(group) action}, we mean a pair $(X,G)$, where $G$ is a countable discrete finitely generated group acting on a compact Hausdorff space $X$ by homeomorphisms.  That is, for every $g \in G$ there is a homeomorphism $\phi_g: X \to X$ such that for any $g_1,g_2 \in G$ and any $x \in X$ we have $\phi_{g_1g_2}(x) = \phi_{g_1}\circ \phi_{g_2}(x)$. Throughout the paper, we use the shortcut notation $gx: = \phi_g(x)$ for the action. Similarly, for a set $A \subset X$ we write $gA:= \{gx : x \in A\}$, for any $g \in G$. The {\it orbit } of $x\in X$ is the set  $O(x)=\{gx: g\in G\}$.  A set $A\subseteq X$ is {\it invariant under the action of} $G$ if $gA=A$, for every $g\in G$. The dynamical system $(X,G)$ is {\it minimal} if $X$ is the only closed non-empty invariant subset of $X$, which is equivalent to saying that $O(x)$ is dense in $X$, for every $x\in X$. An {\it invariant measure} of $(X,G)$ is a Borel probability measure $\mu$ such that $\mu(A)=\mu(gA)$, for every Borel set $A\subseteq X$, and every $g\in G$. The invariant measure $\mu$ is {\it ergodic} if $\mu(A)\in \{0,1\}$, for every invariant Borel set $A\subseteq X$. We denote a dynamical system with an invariant measure $\mu$ by $(X,G,\mu)$.

The \emph{stabilizer} of the action $(X,G)$ at $x\in X$ is a subgroup of $G$ given by 
$$G_x=\{g\in G: gx=x\}.$$ 
Stabilizers of points in the same orbit are conjugate subgroups of $G$. An action $(X,G)$ is \emph{free} if all $x \in X$ have trivial stabilizers. An action $(X,G)$ is \emph{topologically free} if the set of points with trivial stabilizers is residual in $X$. An action $(X,G,\mu)$ is \emph{essentially free} if the set of points with trivial stabilizers has full measure with respect to $\mu$.
 It is immediate that a minimal dynamical system which is essentially free for some measure is topologically free, while there exist examples of minimal topologically free actions which are not essentially free, see \cite{AE2007,Joseph2024,GL2021,BergeronGaboriau2004}. An action $(X,G)$ is \emph{effective} if the only element of $G$ acting trivially on $X$ is the identity $1_G \in G$. 

There are many examples of group actions where the stabilizer of every point is non-trivial, for instance, actions of automorphism groups of rooted trees, see for instance \cite{Grigorchuk2011}. This raises many questions, such as: what subgroups can arise as stabilizers? How are the properties of stabilizers related to the analytic, ergodic or dynamical properties of the action, or to the algebraic properties of the acting group? The interest in this type of questions has been galvanized by the surge of activity in the study of invariant random subgroups (IRS) \cite{AGV2014}, i.e.~probability measures on the space of subgroups of a group $G$ invariant under the natural action of $G$ by conjugation, and by the study of uniform recurrent subgroups (URS), which characterise the set of stabilizers of points with trivial holonomy \cite{GW15,LM18,MT20}.

A simple generalization of an action where a typical point has trivial stabilizer, from the algebraic perspective, is an action where typical points have at most a finite number of distinct stabilizers, which are also finite subgroups of the acting group. This type of action is our main focus in this article.

\begin{defi}\label{almost-normal}
 A subgroup $H$ of a countable group $G$ is said to be \emph{almost normal} if it satisfies the following equivalent conditions:
\begin{enumerate}
    \item The normalizer $N_G(H)$ in $G$ has finite index in $G$.
 \item The subgroup $H$ is a normal subgroup of a subgroup of finite index in $G$.
\item The subgroup $H$ has finitely many conjugate subgroups.
\end{enumerate}
\end{defi}

Before we state our main results in Section \ref{results} we recall a few notions.

\subsection{Actions with no essential holonomy} Following \cite{GL2021,HL2025}, we introduce the notion of a dynamical system with no essential holonomy. This is a property analogous to that of an essentially free system, but in the case when no point of $(X,G,\mu)$ has trivial stabilizer. The idea is the following: instead of only considering whether an element $g \in G$ fixes a point $x \in X$, we look at the behavior of $g \in G$ on a neighborhood of its fixed point $x\in X$. Thus we consider points with trivial or non-trivial holonomy, the notion borrowed from foliation theory. 

\begin{defi}
Let $(X,G)$ be a dynamical system. Then $g \in G$ \emph{has trivial holonomy at} $x \in X$ if $g\in G_x$ and there exists an open neighborhood $V \owns x$ such that $g|V = {\rm id}$. If such a neighborhood does not exist, then $g$ \emph{has non-trivial holonomy at} $x$.

Moreover, a point $x \in X$ \emph{has trivial holonomy}, or \emph{is a point without holonomy}, if every $g \in G_x$ has trivial holonomy at $x$. If there exists $g \in G_x$ which has non-trivial holonomy at $x \in X$, then $x$ \emph{has non-trivial holonomy}, or \emph{is a point with holonomy}. 
\end{defi}

 It is known that the set of points with trivial holonomy is residual in $X$ \cite{EMT1977}, that is, points without holonomy are topologically typical. The situation is more complicated in the measure-theoretical setting. 

\begin{defi}\label{defi-essential}
    Let $(X,G, \mu)$ be an ergodic measure-preserving dynamical system with probability measure $\mu$. Then $(X,G,\mu)$ has \emph{no essential holonomy} with respect to the invariant measure $\mu$ if $\mu$-a.e.~point in $X$ has trivial holonomy. 
\end{defi}
 
 If $(X,G,\mu)$ in Definition \ref{defi-essential} is essentially free, then $x \in X$ has non-trivial stabilizer if and only if it has non-trivial holonomy \cite[p.~2003]{GL2021}. Thus a topologically free action with no essential holonomy is an essentially free action. 

\subsection{Locally quasi-analytic actions}\label{LQA-actions} Systems where every point has a non-trivial stabilizer can be divided into two classes: locally quasi-analytic actions and non-locally quasi-analytic actions. Locally quasi-analytic actions were introduced in \cite{ALC2009}, generalizing the notion of quasi-analytic actions on manifolds \cite{H1985} to topological spaces.

Let $d$ be a metric on $X$, compatible with the topology on $X$ (this will be our standing assumption on the metric in the rest of the paper). An action $(X,G)$ is \emph{locally quasi-analytic (LQA)} if there exists $\epsilon >0$ such that for all open $U \subset X$ with $\diam (U) < \epsilon$ and all $g \in G$, if $g|V = {\rm id}$ for some open $V \subset U$, then $g|U = {\rm id}$. Equivalently, if the actions of  $g_1,g_2 \in G$ coincide on an open subset $V$ of $U$, then they coincide on $U$. Thus an action is locally quasi-analytic, if the action of any $g \in G$ has unique extensions from sets of diameter less than $\epsilon$ to sets of diameter $\epsilon$. Locally quasi-analytic actions are `well-behaved' and this property, in particular, has a relation with the problem of realizability of group actions on topological spaces as transversal dynamical systems in real-analytic foliations of manifolds, see for instance \cite[Corollary 1.12]{HL2019}.

If an action $(X,G)$ is not locally quasi-analytic, then for every $\epsilon >0$ there is an open set $U \subset X$ with $\diam(U)< \epsilon$, and an open proper subset $V \subset U$ and $g \in G$ such that $g|V = {\rm id}$ and $g|U \ne {\rm id}$. If, in addition, $(X,G)$ is minimal, then this property holds for any open set $U$.
A related notion of a \emph{micro-supported} action was considered in \cite{LM18}. A micro-supported action is ~an effective action $(X,G)$ such that for every non-empty open set $U \subset X$ the set $G_U$ of elements $g \in G$, such that $g$ fixes every point in $X\backslash U$, is non-trivial. A micro-supported action $(X,G)$ is not locally quasi-analytic, see Section \ref{micro-supported}.

\subsection{URS and IRS} Let $S(G)$ be the space of subgroups of $G$ with Chaubaty-Fell topology, see Section \ref{sec-inv-random} for details. The group $G$ acts on $S(G)$ continuously by conjugation. Given any dynamical system $(X,G,\mu)$, there is a map ${\rm Stab}: X \to S(G): x \mapsto G_x$ which is continuous at the points with trivial holonomy, see \cite{Vorobets2012,GW15}. A \emph{uniformly recurrent subgroup (URS)}, introduced in \cite{GW15}, is a minimal closed invariant subset of $S(G)$ for the action by conjugation. An \emph{invariant random subgroup (IRS)} is an invariant probability measure on $S(G)$ \cite{AGV2014}.   The \emph{stabilizer URS} of a minimal dynamical system $(X,G)$ is the closure of the set of stabilizers of points with trivial holonomy in $S(G)$; it is the unique closed minimal invariant subset of the topological closure $\overline{{\rm Stab}(X)}$ in $S(G)$ (see \cite{GW15}). 
The IRS of a dynamical system $(X,G,\mu)$ is the push-forward measure $\nu = {\rm Stab}_*\mu$. Thus the question of when a topologically free action is essentially free, or, more generally, when a minimal dynamical system $(X,G, \mu)$ has no essential holonomy, can be rephrased in terms of URS and IRS. Namely, a minimal dynamical system $(X,G, \mu)$ has no essential holonomy if and only if the stabilizer URS is the support of the IRS of the dynamical system $(X,G,\mu)$.
 
\subsection{Main results}\label{results} We are now ready to state our main results. Our first theorem provides equivalent conditions for an action to have a finite almost normal subgroup as a topologically typical stabilizer.

 \begin{theo}\label{theo-1}
     Let $(X,G)$ be a minimal group action. Then the following statements are equivalent:
     \begin{enumerate}
         \item There exists a point $x \in X$ with finite stabilizer.
         \item There exists an almost normal finite subgroup $H$ of $G$ such that $x \in X$ has trivial holonomy if and only if  $G_x = gHg^{-1}$, for some $g \in G$.  
         \item The stabilizer URS of $(X,G)$ is a finite set of subgroups in $S(G)$.  
         \item Suppose, in addition, $X$ has a metric $d$. Then the action $(X,G)$ is locally quasi-analytic, and there exists a point $x \in X$ with a finite stabilizer.
     \end{enumerate}
  \end{theo}   

  The equivalence of statement (4) to statements (1) - (3) above is a consequence of a more general statement below, which does not require finiteness of stabilizers.

\begin{theo}\label{theo-LQA1}
    Let $(X,G)$ be a minimal dynamical system, and $X$ be a metric space. Then $(X,G)$ is locally quasi-analytic if and only if the stabilizer URS of $(X,G)$ is a finite set of conjugate subgroups if and only if the stabilizers of points with trivial holonomy are conjugate subgroups of an almost normal subgroup $H$ of $G$.
\end{theo}

An analog of Theorem \ref{theo-LQA1} in a more restrictive setting of equicontinuous group actions on Cantor sets was proved in \cite[Theorem 1.7]{GL2021}. 

Recall that $G$ is \emph{allosteric} if there exists a minimal topologically free action of $G$ on a compact Hausdorff space which is not essentially free w.r.t.~an ergodic measure. Further, recall that groups $G_1$ and $G_2$ are \emph{commensurable} if there are finite index subgroups $H_1 < G_1$ and $H_2<G_2$ which are isomorphic. By \cite[Theorem 2.9]{Joseph2024} the relation of being allosteric is preserved under commensurability of groups. In Corollary \ref{allosetric-sufficient} we prove that allostery is preserved under quotients by finite normal subgroups, which leads to the following result.

\begin{theo}\label{coro-IRS}
If a minimal ergodic action $(X,G,\mu)$ satisfies one of the conditions of Theorem \ref{theo-1} with group $H$ as in Theorem \ref{theo-1},(2), and, in addition, $G$ is not allosteric and $N_G(H)/H$ is commensurable to $N_G(H)$, then: 
\begin{enumerate}
    \item The IRS associated to $(X,G,\mu)$ is an atomic probability measure on $S(G)$.
    \item The IRS associated to $(X,G,\mu)$ is supported on the stabilizer URS of $(X,G,\mu)$, i.e.~$(X,G,\mu)$ has no essential holonomy.
\end{enumerate}
 \end{theo}

Residually finite groups which are not allosteric satisfy the commensurability condition above, see Corollary \ref{finite-summarize-RF}, and we obtain the following corollary.

\begin{coro}
    Let $G$ be a residually finite group. If $G$ is not allosteric, then every ergodic action $(X,G,\mu)$ with finite stabilizers has no essential holonomy.
\end{coro}

Next we are interested in the question how actions with almost normal finite stabilzers arise. Our next result shows that such actions can be constructed as finite factors of essentially free systems. Recall that $N_G(H)$ denotes the normalizer of $H$ in $G$.

\begin{theo}\label{characterization-factor}
Let $(X,G)$ be a minimal dynamical system, and let $H$ be a finite subgroup of $G$. Then the following is true.
\begin{enumerate}
\item If $(X,G)$ is topologically free, and $H$ is  almost normal,  then the following sentences are equivalent:

\medskip

\begin{itemize}
\item[(1.1)] There exists a factor map $\phi:X\to G/N_G(H)$,  where $G$ acts on $G/N_G(H)=\{gN_G(H): g\in G\}$ by left multiplication.

\item[(1.2)] There exist a minimal dynamical system $(Y,G)$,  and a $|H|$-to-one factor map $\pi:X\to Y$ such that 
$$
\pi^{-1}Y_H=\{x\in X: G_x=\{1_G\}\},
$$
where $Y_H$ is the set of all $y\in Y$ such that $G_y$ is conjugate to $H$.  In this case  ${\rm Conj}(H)$ is the stabilizer URS of $(Y,G)$, and, if $Y$ is a metric space, $(Y,G)$ is locally quasi-analytic.
\end{itemize}

\item If $(X,G,\mu)$ is essentially free with respect to an ergodic measure $\mu$,  then the following sentences are equivalent:

\medskip

\begin{itemize}
\item[(2.1)] $H$ is almost normal and there exists a factor map $\phi:X\to G/N_G(H)$.

\item[(2.2)] There exist a minimal dynamical system $(Y,G)$,  and a $|H|$-to-one factor map $\pi:X\to Y$ such that 
$$
\pi^{-1}Y_H=\{x\in X: G_x=\{1_G\}\}.
$$

\item[(2.3)] There exist a minimal  dynamical system $(Y,G)$ and a $|H|$-to-one $\pi_*\mu$-a.e.~ factor map $\pi:X\to Y$ such that $\pi_*\mu(Y_H)=1$. In this case, ${\rm Conj}(H)$ is the stabilizer URS of $(Y,G,\pi_*\mu)$, and its IRS is supported on its stabilizer URS. In particular, $(Y,G,\pi_*\mu)$ has no essential holonomy, and, if $Y$ is a metric space, $(Y,G)$ is locally quasi-analytic.
\end{itemize}
\end{enumerate}
\end{theo}

Finally, we ask whether all actions where almost every point has an almost normal  finite stabilizer arise as factors with finite fibers of essentially free actions. For a group $G$, denote by $R(G)$ the intersection of all finite index normal subgroups in $G$. Recall that $Y_H$, for a subgroup $H$ of $G$, denotes the set of all $y\in Y$ such that $G_y$ is conjugate to $H$. 

\begin{prop}\label{theo-main-res-finite}
   Suppose that there exists an almost normal finite subgroup $H$ of $G$ such that $H\cap R(G)=\{1_G\}$. Let $(Y,G)$ be a dynamical system such that $Y_H\neq\emptyset$. Then there exist a dynamical system $(X,G)$, a factor map $\pi:X\to Y$, and a constant $m\geq 1$ such that
  $$
  \pi^{-1}Y_H\subseteq\{x\in X: G_x=\{1_G\}\},$$  
  and
  $$
 |\pi^{-1}\{y\}|\leq m, \mbox{ for all } y\in Y_H.
 $$
Moreover,
\begin{enumerate}
\item If $(Y,G)$ is minimal, then $(X,G)$ is minimal and topologically free.

\item If there exists an ergodic measure $\mu$ of $(Y,G)$ such that $\mu(Y_H)=1$, then $(X,G)$ is essentially free with respect to a lift of $\mu$. 
\end{enumerate}
\end{prop}

Theorem \ref{characterization-factor} and Proposition \ref{theo-main-res-finite} show that the property of having no essential holonomy is preserved under finite factors, and almost finite-to-one extensions of dynamical systems.

In particular, the hypothesis of Proposition \ref{theo-main-res-finite} is satisfied if the group $G$ is residually finite. We give a few extra conditions under which Proposition \ref{theo-main-res-finite} holds in Proposition \ref{sufficient-condition}, Lemma \ref{sufficient-condition-1} and Corollary \ref{quasi-general-case}. Examples of families of actions with almost normal stabilizers are given in Examples \ref{ex-farber} and \ref{FC-central}.

The rest of the paper is organized as follows. In Section \ref{preliminaries} we recall the basic notions from dynamics and group theory. In Section \ref{sec-holonomy} we prove Theorems \ref{theo-1}, \ref{coro-IRS}, and \ref{theo-LQA1}. In Section \ref{sec-factors} we study the factors of essentially free systems and prove Theorem \ref{characterization-factor}. In Section \ref{sec-extensions} we prove Proposition \ref{theo-main-res-finite}.

\section{Preliminaries}\label{preliminaries}

\subsection{Factors and extensions} Let $(X,G)$ and $(Y,G)$ be dynamical systems. A {\it factor map} from $(X, G)$ to $(Y,G)$ is a continuous surjective map $\pi:X\to Y$ such that $\pi(gx)=g\pi(x)$, for every $x\in X$. If there is $n\geq 1$ such that $|\pi^{-1}\{y\}|=n$, for every $y\in Y$, we say that $\pi$ is {\it $n$-to-one}.  If $\mu$ is an invariant measure of $(X,G)$, the pushforward $\pi_*\mu$  of $\mu$  is an invariant measure of $(Y,G)$. If the set of $y\in Y$ such that $|\pi^{-1}\{y\}|<\infty$ has full measure with respect to an invariant measure $\nu$ of $(Y,G)$, we say that $\pi$ is {\it finite-to-one $\nu$-a.e.} If $\nu$ is an invariant measure of $(Y,G)$, a {\it lift} of $\nu$ is an invariant measure $\mu$ of $(X,G)$ such that $\pi_*\mu=\nu$.

\subsection{Micro-supported actions are not locally quasi-analytic}\label{micro-supported} Definitions of a locally quasi-analytic and of a micro-supported action are given in Section \ref{LQA-actions}. 

\begin{prop}
    If an effective action $(X,G)$ is micro-supported, then it is not locally quasi-analytic.
\end{prop}

\begin{proof}
    Suppose $(X,G)$ is micro-supported, and let $W \subset X$ be an open set. Note that since a micro-supported action is effective by definition, then $X$ does not have isolated points, and so $W$ is not a singleton.  Choose a non-empty open subset $V \subset W$ such that the topological closure $\overline V$ is properly contained in $W$. Then $W \backslash \overline{V}$ is open, and by the definition of a micro-supported action there exists $g \ne 1_G$ such that $g|X \backslash (W \backslash \overline{V}) = {\rm id}$. Then $g|V = {\rm id}$ and $g|W \ne {\rm id}$, since $g \ne 1_G$ and so it must act non-trivially on $W \backslash \overline{V}$.  
\end{proof}

\subsection{Uniformly recurrent subgroups and invariant random subgroups}\label{sec-inv-random}
Let $S(G)$ be the set of all subgroups of $G$. The Chabauty-Fell topology \cite{GW15,LM18,MT20,Vorobets2012} on $S(U)$ has a basis of open sets 
  $$U(A,B) = \{H \in S(G): A \subset H \textrm{ and }B\cap H = \emptyset\},$$
where $A,B$ are finite subsets of $G$. With this topology, $S(G)$ is a compact metrizable space. The action of $G$ on $S(G)$ by conjugation is continuous, which implies that $(S(G), G)$ is a dynamical system. A \emph{uniformly recurrent subgroup (URS)}, introduced in \cite{GW15}, is a minimal closed invariant subset of $S(G)$ for the action by conjugation. An \emph{invariant random subgroup (IRS)} is an invariant probability measure on $S(G)$ \cite{AGV2014}.

If $(X,G)$ is a dynamical system, then the map $\Stab: X\to S(G)$ given by $\Stab(x)=G_x$, is upper-semi-continous and equivariant (see, for example, \cite{GW15}). The map $\Stab$ is Borel, and it is continuous precisely at points with trivial holonomy \cite[Lemma 5.4]{Vorobets2012}. Since $\Stab$ is a Borel map,  for every invariant measure $\mu$ of $(X,G)$, the pushforward  $\Stab_*\mu$ is an invariant measure of $(S(G),G)$.  Glasner and Weiss have shown in \cite{GW15} that if the system $(X,G)$ is minimal, then there is a unique minimal set in $(\overline{\Stab(X)},G)$ (see also \cite{MT20}). This minimal set is called the {\it stabilizer URS} of $(X,G)$, and it is precisely the topological closure of the image in $S(G)$ of the set of points with trivial holonomy in $(X,G)$, see \cite{GW15} or \cite[Lemma 2.2, Proposition 2.5]{LM18}. If $(X,G,\mu)$ has no essential holonomy, then the pushforward measure $\Stab_*\mu$ is supported on the stabilizer URS of the action. This gives an alternative definition of an action with no essential holonomy: a minimal action $(X,G,\mu)$ has no essential holonomy if and only if the IRS associated to $(X,G,\mu)$ is supported on its URS. It was proved in \cite[Proposition 13]{AGV2014} that, given an IRS $\nu$ on $S(G)$, there exists a measure-preserving action $(X,G,\mu)$ on a Borel probability space such that $\nu = \Stab_* \mu$. Similarly, it was proved in \cite[Corollary 1.2]{MT20} that every URS in $(S(G),G)$ can be realized as a stabilizer URS of a minimal action of $G$ on a compact space $X$.

\subsection{$G$-odometers}\label{odom-action} 
We describe a method to construct a group action on a Cantor set which we use in a few of our examples. The action constructed below is called a \emph{$G$-odometer}, where $G$ is the acting group (see, for example,  \cite{CortezPetite2008}).

Let $G$ be a finitely generated residually finite group and consider a decreasing infinite sequence of finite index subgroups $\mathcal G: G=G_0 \supset G_1 \supset \cdots$, that is, for any $i \geq 1$ the inclusion $G_i \subset G_{i-1}$ is proper. We do not require the subgroups to be normal, and we do not require that the intersection $K(\mathcal G) = \bigcap G_i$ is trivial. 
Every such group chain $\mathcal G$ gives rise to an action of $G$ on a Cantor set as follows: for $i \geq 0$ the coset space $G/G_i$ is a finite set, and there are maps $f^{i+1}_{i}: G/G_{i+1} \to G/G_i$ given by coset inclusions $gG_{i+1} \subset gG_{i}$. Note that $|G/G_{i+1}| = |G/G_i||G_i/G_{i+1}|$ for all $i \geq 0$. The group $G$ acts on the coset spaces $G/G_i$ by left translations, permuting the cosets, and preserving the relation of coset inclusions. Since the group chain is decreasing, the inverse limit space
  $$X = \lim_{\longleftarrow}\{f^{i+1}_i: G/G_{i+1} \to G/G_i\}: = \{(x_0, x_1, \ldots ): f^{i+1}_i(x_{i+1}) =x_i, i\geq 0\} \subset \prod_{i \geq 0} G/G_i$$
is a Cantor set with respect to the relative topology induced by the product topology on the infinite product $\prod_{i \geq 0} G/G_i$.  Basic sets in $X$ in this topology are the sets
  $$U_{g,i} = \{(x_0, x_1,\ldots) \in X : x_i = gG_i\},$$
  for $g \in G$ and $i \geq 0$. These sets are
also called cylinder sets. The group $G$ acts on each coset $G/G_i$ by left multiplication, and this action of $G$ preserves inclusions of cosets. Thus there is the induced action on the inverse limit space
  \begin{align}\label{eq-Godomaction}
      G \times X \to X: (g, (x_0,x_1,\ldots)) \mapsto (gx_0,gx_1,\ldots).
  \end{align}
Since the action of $G$ on each coset space $G/G_i$ is transitive, $(X,G)$ is minimal. The standard metric on $X$ is given by, for any $\overline{x} = (x_0,x_1,\ldots)$ and any $\overline{y} = (y_0,y_1,\ldots)$ in $X$,
  $$d(\overline{x}, \overline{y}) = \frac{1}{2^n}, \quad n = \sup \{i \geq 0 : x_i = y_i\}.$$
With respect to this metric, the action in \eqref{eq-Godomaction} is isometric. The action $(X,G)$ is uniquely ergodic with the counting measure $\mu$ on $X$ defined by assigning, for each cylinder set $U_{g,i}$, $i \geq 0$, $g \in G$,
  $$\mu(U_{g,i}) = \frac{1}{|G:G_i|}.$$
We note that every $G$-odometer action arising from a group chain can be represented as an action on the boundary of a rooted spherically homogeneous tree \cite[p.~2012]{GL2021} and vice versa, providing a large source of examples of $G$-odometer actions. Also, every minimal equicontinuous action $(X,G)$ on a Cantor set is conjugate to a $G$-odometer action, although the choice of a group chain $\mathcal G$ in this case is not unique, see \cite{DHL2016} for more details.

\section{Holonomy and almost normal stabilizers}\label{sec-holonomy}\label{holonomy-stabilizers}

In this section we prove Theorem \ref{theo-1} and Theorem \ref{coro-IRS}. For a subgroup $H$ of $G$, denote by $\Conj(H)$ the set of conjugates of $H$ in $G$.

Let $(X,G)$ be a dynamical system. For every subgroup $H$ of $G$ we define
$$
X_H=\{x\in X: G_x\in \Conj(H)\},
$$
and
$$
X_H(g)=\{x\in X: G_x=gHg^{-1}\}, \mbox{ for every } g\in G.
$$
The set $X_H$ is invariant under the action of $G$ and is Borel. Indeed, $X_H$ is the pre-image, by the function $\Stab:X\to S(G)$, of an orbit of $G$ in the space of subgroups $S(G)$ under the action of $G$ by conjugation, see Section \ref{sec-inv-random}.

\subsection{Some properties of almost normal stabilizers} In this section we collect a few technical lemmas used in the proofs of our results. While the proofs are rather straightforward, we include them here for completeness.

\begin{lemma}\label{lemma_contention}
Let $(X,G)$ be a dynamical system. Let $H \subset G$ be an almost normal subgroup of $G$, and suppose the set $X_H$ defined above is non-empty. Then  for every $x\in \overline{X_H}$ there exists $g\in G$ such that $x \in \overline{X_H(g)}$, and then $gHg^{-1}\subseteq G_x$. 
\end{lemma}

\begin{proof}
    Since $\Conj(H)$ is finite there exists $g\in G$ such that $x\in \overline{X_H(g)}$. Since $G$ acts on $X$ by homeomorphisms, the continuity of the action implies that $gHg^{-1}\subseteq G_x$.
\end{proof}

For not almost normal groups, it is easy to find examples of subgroups which contain their conjugate as a proper subgroup. For instance, for the Baumslag-Solitar group $BS(1,n) = \{a,b: bab^{-1} = a^n\}$, $n \geq 2$, we have $b\langle a \rangle b^{-1} = \langle a^n \rangle \subsetneq \langle a \rangle$.  We show that this situation does not happen for almost normal groups, which has implications on the topological properties of certain dynamically defined sets in $X$ in Theorem \ref{converse}.

\begin{lemma}\label{auxiliar}
Let $H$ be an almost normal subgroup of $G$. If $g,l\in G$ are such that $gHg^{-1}\subseteq lHl^{-1}$, then $gHg^{-1}=lHl^{-1}$.    
\end{lemma}
\begin{proof}
If there exists $s\in G$ such that $sHs^{-1} \subsetneq  H$, a proper subgroup of $H$, then for $k\geq 1$, $$s^kHs^{-k} \subsetneq  s^{k-1}Hs^{-(k-1)},$$ 
and we obtain a descending infinite sequence of proper subgroups conjugate to $H$. But this is not possible since $H$ is almost normal, and so has finitely many distinct conjugate subgroups.  
\end{proof}

 \subsection{Dynamically defined partitions} In this section we prove that an action with almost normal stabilizers defines a covering of $\overline X_H$, which becomes a partition if $(X,G)$ is minimal. The existence of such partitions provides a useful tool for the proofs of our theorems. 

\begin{theo}\label{converse}
Let $(X,G)$ be a dynamical system. Suppose there exists an almost normal subgroup $H$ of $G$ such that $X_H\neq \emptyset$.  
If $g_1,\cdots, g_m$ are representing elements of each class in $G/N_G(H)$,
 then the collection $$\C=\{\overline{X_H(g_i)}: 1\leq i\leq m\}$$ is a closed covering of $\overline{X_H}$ that verifies the following:
\begin{enumerate}
\item For every $p_1,p_2\in G$  with $p_1\in p_2N_G(H)$, and for every $h\in H$, we have
$$
p_2\overline{X_H(h)}=\overline{X_{H}(p_2)}=\overline{X_H(p_1)}.
$$

\item If $$E=\bigcup_{1\leq i\neq j\leq m}\overline{X_{H}(g_i)}\cap \overline{X_{H}(g_j)}$$
is non-empty, then $E$ is a closed invariant set contained in $X_H^c$.

\item If $X_H$ is dense in $X$, then
$$
X_H = \{ x\in X: x \textnormal{ has trivial holonomy}\}.
$$ 
\item If the action $(X,G)$ is minimal, then $E$ is the empty set, and ${\rm Conj}(H)$ is the stabilizer URS of $(X,G)$.

 \end{enumerate}
 \end{theo}
\begin{proof}
By definition,  $g\in lN_G(H)$ if and only if $X_H(g)=X_H(l)$. Thus, if $g\in lN_G(H)$ then $\overline{X_H(g)}=\overline{X_H(l)}$. This implies that $\{\overline{X_{H}(g_i)}: 1\leq i\leq m\}$ is a closed covering of $\overline{X_H}$. Since for every $g\in G$ we have $gX_H(1_G)=X_H(g)$, the continuity of the action implies $g\overline{X_{H}(1_G)} = \overline{X_H(g)}$, and (1) follows.

\medskip

Statement (2) will be a consequence of the following claim:

{\it Claim 1:} If $x\in X$ is such that $\overline{O(x)}\cap X_H\neq \emptyset$, then $gHg^{-1}\cup lHl^{-1}\subseteq G_x$ implies $gHg^{-1}=lHl^{-1}$.   

\begin{proof}[Proof of Claim 1] Let $x\in X$ be such that   $\overline{O(x)}\cap X_H\neq \emptyset$. If $O(x) \cap \overline{X_H} \ne \emptyset$, then we are done, so assume that $\overline{X_H}$ contains an accumulation point of $O(x)$. This does not imply that $O(x)$ intersects $\overline{X_H}$, so we cannot use Lemma \ref{lemma_contention}. Instead, we are going to find a point $y \in \overline{O(x)}\cap X_H$ such that $G_y$ contains $gHg^{-1} \cup lHl^{-1}$, for some $g,l \in G$. Then, since $y \in X_H$, there is $m \in G$ such that $(gHg^{-1}\cup lHl^{-1} )\subset m H m^{-1}$, and Lemma \ref{auxiliar} implies that $lHl^{-1} = mHm^{-1} = gHg^{-1}$.

Since $N_G(H)$ is a finite index subgroup of $G$, its core $\Gamma=\bigcap_{g\in G}gN_G(H)g^{-1}$ is a normal finite index subgroup of $G$.      Since $\Gamma$ is a finite index subgroup of $G$, for  $y\in \overline{O(x)}\cap X_H$ there exists $d\in G$ such that for every open neighborhood $V$ of $y$, there exists $\gamma_V\in\Gamma$ such that $d\gamma_Vx\in V$. By continuity of the action, this implies that for every neighborhood $U$ of $d^{-1}y$, there exists $\gamma_U\in\Gamma$ such that 
\begin{equation}\label{eq1}
\gamma_Ux\in U.
\end{equation}
Let $g, l \in G$ be such that $gHg^{-1} \cup l H l^{-1}\subseteq G_x$. For every $h\in H$ and every open neighborhood $W$ of $ghg^{-1}d^{-1}y$, there exists a neighborhood $U_W$ of $d^{-1}y$ and $\gamma_{U_W} \in \Gamma$ such that
\begin{equation}\label{eq2}
ghg^{-1}\gamma_{U_W}x\in W. 
\end{equation}

Observe that if $d^{-1}y\neq ghg^{-1}d^{-1}y$, then we can take $W\cap U_{W}=\emptyset$. 

Due to $\Gamma \subseteq N_G(gHg^{-1})$, for every $W$ we have
$$
\gamma_{U_W}^{-1}ghg^{-1}\gamma_{U_W}x=x,
$$
since $g H g^{-1} \subset G_x$ by assumption.
From this we get
$$
ghg^{-1}\gamma_{U_W}x=\gamma_{U_W}(\gamma_{U_W}^{-1}ghg^{-1}\gamma_{U_W})x=\gamma_{U_W}x.
$$
Thus, from equations (\ref{eq1}) and (\ref{eq2}), we obtain
$$
ghg^{-1}d^{-1}y=d^{-1}y,
$$
so $gHg^{-1} \subset G_{d^{-1}y}$.
Repeating the above argument for $l H l^{-1}$ and using the fact that $\Gamma$ also normalizes $lHl^{-1}$, we obtain $lHl^{-1} \subseteq G_{d^{-1}y}$. Thus $d^{-1}g Hg^{-1}d \cup d^{-1}l H l^{-1}d \subset G_y$, where $y \in X_H$, and the argument at the beginning of the proof shows Claim 1. \end{proof}

\medskip  

Since by Lemma \ref{lemma_contention} $y\in \overline{X_{H}(g_i)}$ implies $g_iHg_i^{-1}\subseteq G_y$, by Claim 1 we deduce that if $y\in \overline{X_{H}(g_i)}\cap \overline{X_{H}(g_j)}$, for some $1\leq i\neq j \leq m$, then $\overline{O(y)}\cap X_H=\emptyset$, and so $E \subset X_H^c$, the complement of $X_H$ in $X$. The set $E$ is closed since it is a finite intersection of closed sets, and it is invariant since the stabilizers of points in the same orbit are conjugate. This shows statement (2).

\medskip

For (3), let $x\in X_H$, and let $1\leq j\leq  m$ be such that $G_x=g_jHg_j^{-1}$.  Choosing an open neighborhood $V$ of $x$ small enough, we can arrange that for every $y\in V$, the group $g_jHg_j^{-1}$ is contained in $G_y$. Indeed, if such $V$ does not exist, then by Lemma \ref{lemma_contention} we should have $G_x = g_jHg_j^{-1} \supset g_i H g_i^{-1}$ for $g_i \ne g_j$, which is not possible by Lemma \ref{auxiliar}. Thus the action of every $h \in g_jHg_j^{-1}$ fixes $V$, which shows that $x$ has trivial holonomy.  

Now, suppose that $x\in X$ has trivial holonomy. Since $X=\overline{X_H}$, there exists $1\leq i \leq m$ such that $x\in \overline{X_H(g_i)}$. This implies that $g_iHg_i^{-1}\subseteq G_x$, and we must show the reverse inclusion. For that, by an argument similar to the one in the previous paragraph, we can find an open neighborhood $V$ of $x$ such that $X_H \cap V = X_H(g_i) \cap V$, i.e.~for any $y \in V$ we have $G_y = g_iHg_i^{-1}$. Let $g \in G_x$. Since $x$ has trivial holonomy, there is an open neighborhood $W_g \subset V$ such that $g$ fixes every point in $W_g$, and in particular the points in $W_g \cap X_H(g_i)$. This shows that $G_x \subset g_i H g_i^{-1}$, and $x \in X_H$.

\medskip
 To show (4), note that if $(X,G)$ is minimal, then $X = \overline{X_H}$, and by Lemma \ref{lemma_contention} for any $y \in X$ the stabilizer $G_y$ contains at most one distinct conjugate of $H$. It follows that $E$ must be empty.

Next, by (3) $X_H$ coincides with the set of points with trivial holonomy of $(X,G)$, and then (4) follows by the result of Glasner and Weiss \cite{GW15}, see also \cite[Lemma 2.2, Proposition 2.5]{LM18}, since $(X,G)$ is minimal.
\end{proof}  

\begin{remark}\label{transitive-example}
{\rm For the point (3) of Theorem \ref{converse} to hold, it is necessary that the set $X_H$ is dense in $X$, as the following example shows.  
Consider the space $X=\{0,1\}^{\ZZ}$ of bi-infinite sequences of $0$'s and $1$'s equipped with product topology. Then $X$ is a Cantor set, i.e.~a compact metrizable totally disconnected perfect space. Define an action of $\ZZ$ on $X$ by a shift transformation $\sigma:X\to X$, i.e.~for $\underline{x} = (x_i)_{i \in \mathbb Z}$ we have $\sigma(\underline{x})_i = x_{i+1}$. Let $H=n\ZZ$, for some fixed integer $n>1$. In this case, $X_H=\{x\in \{0,1\}^{\ZZ}: \sigma^n(\underline x)=\underline x\}$ is a finite set and therefore is not dense in $\{0,1\}^{\ZZ}$. On the other hand, the action $(X,\mathbb Z)$ has an element with a dense orbit (an infinite sequence obtained by concatenating first all words of length $1$, then all words of length $2$, and so on), the points of which have trivial stabilizers. Thus the elements of $X_H$ are approximated with an orbit with trivial stabilizer and, in particular, points in $X_H$ do not have open neighborhoods fixed by $\sigma^n$ and do not have trivial holonomy. A similar argument applies to any point with non-trivial stabilizer, and so the set of points with trivial holonomy is precisely the set of points with trivial stabilizers. This example also shows that minimality is essential for the uniqueness of the URS in statement (4) of Theorem \ref{converse}. Indeed, for $H=\{1_G\}$ the set $X_H$ is dense, nevertheless, every subgroup of $\ZZ$ defines a minimal closed invariant set of $\overline{\Stab(X)}$. }    
\end{remark}

\begin{coro}\label{converse-minimal}
Let $(X,G)$ be a minimal dynamical system. Suppose there exists an almost normal subgroup $H$ of $G$ such that $X_H\neq \emptyset$. Then, if $g_1,\cdots, g_m$ are representatives of the cosets in $G/N_G(H)$, the collection $$\C=\{\overline{X_H(g_i)}: 1\leq i\leq m\}$$ is a clopen partition of $\overline{X_H}$ that induces a factor map $\phi:X\to G/N_G(H)$ given by 
$$
\phi(x)=gN_G(H) \mbox{ if and only if } x\in \overline{X_H(g)}. 
$$

\end{coro}
\begin{proof}
   \medskip
Since $X_H$ is invariant, the minimality of $(X,G)$ implies that $X=\overline{X_H}$. Theorem \ref{converse}, (2), implies that $E=\emptyset$, and the collection $\C$ is a clopen partition of $X$. Theorem \ref{converse}, (1), implies that the map $\phi:X\to G/N_G(H)$ is well defined, continuous, surjective and commutes with the respective actions. 
\end{proof}

\subsection{Finite  stabilizers and properties of actions}

 In this section, we study the relationship between holonomy, finite stabilizers and the existence of actions with certain properties, proving Theorems \ref{theo-1} and \ref{theo-LQA1}. The proof is divided into Proposition \ref{equivalence-URS}, for statements (1) - (3) of Theorem \ref{theo-1}, and Theorem \ref{theo-LQA}, from which we then derive Theorem \ref{theo-1},(4). Recall that $\Conj (H)$ denotes the set of conjugates of a subgroup $H$ of $G$.
 
Proposition \ref{equivalence-URS} is in line with similar results on the characterization of topologically free minimal systems in other contexts (see, for example,  \cite[Lemma 2.1]{Joseph2024}). 

\begin{prop}\label{equivalence-URS} 
Let $(X,G)$ be a minimal system. The following statements are equivalent:
\begin{enumerate}
\item There exists $x\in X$ with finite stabilizer.

\item There exists an almost normal finite subgroup $H$ of $G$, such that the stabilizer URS of $(X,G)$  is the orbit ${\rm Conj}(H)$ of $H$ in $S(G)$.

\item There exists an almost normal finite subgroup $H$ of $G$ such that $X_H=\{x\in X: x \mbox{ has trivial holonomy}\}$. 

\item There exists an almost normal finite subgroup $H$ of $G$ such that $X_H\neq \emptyset$.
\end{enumerate}
\end{prop}
\begin{proof}

  Suppose there exists a finite subgroup $K\in \Stab(X)$ (see Section~\ref{sec-inv-random} for the definitions of the space of subgroups $S(G)$ and the map ${\rm Stab}:X\to S(G)$). Then there exists a finite group $H\in \overline{\Stab(X)}$ with minimal cardinality (that is, $|H|\leq |H'|$ for every $H'\in \overline{\Stab(X)}$).   Consider the set $A=\overline{\{gHg^{-1}: g\in G\}}\subseteq \overline{\Stab(X)}$.  The definition of the Chabauty--Fell topology for countable groups implies that if $L\in A$, then for every finite set $F\subseteq L$ there exists $g\in G$ such that $gHg^{-1}\cap F=L\cap F=F$, which implies that $|L|\leq |H|$. Because the cardinality of $H$ is minimal, it follows that if $L\in A$ then $|L|=|H|$, and hence $L=gHg^{-1}$ for some $g\in G$. Thus $A$ is the collection of conjugates of $H$, that is, a single orbit in $S(G)$. Since $H$ is finite and $A$ is compact, $A$ must be finite. Indeed, if $A$ were infinite, there would exist $L\in A$ which is an accumulation point of a sequence $(g_iHg_i^{-1})_{i\in\mathbb{N}}$ with $g_iHg_i^{-1}\neq g_jHg_j^{-1}$ for every $i\neq j$. Since $|L|=|H|$, there exists $i_0$ such that $g_iHg_i^{-1}=L$ for every $i\geq i_0$, which is a contradiction. Since $A$, the orbit of $H$ in $S(G)$, is finite, it follows that $H$ is almost normal, and its orbit is a minimal subset of $\overline{\Stab(X)}$. By uniqueness of the stabilizer URS for minimal systems, it follows that $A=\Conj(H)$ is the stabilizer URS of $(X,G)$. Thus (1) holds if and only if (2) holds. Statement (3) is equivalent to (2) by the definition of a URS, and (4) follows from (3) since the set of points without holonomy is residual (and therefore non-empty) in $X$.
\end{proof}

\begin{theo}\label{theo-LQA}
Let $(X,G)$ be a minimal dynamical system. Then the following are equivalent:
    \begin{enumerate}
    \item There exists an almost normal subgroup $H$ of $G$ such that $$X_H=\{x\in X: x \mbox{ has trivial holonomy}\}.$$  
        \item The stabilizer URS of $(X,G)$ is a finite set of conjugate subgroups.
        \item If in addition $X$ has metric $d$, the action $(X,G)$ is locally quasi-analytic.
    \end{enumerate}
\end{theo}

In the case when $X$ is a metric space, Theorem \ref{theo-LQA} directly implies Theorem \ref{theo-LQA1}, which states that stabilizers of points with trivial holonomy of locally quasi-analytic actions are always almost normal subgroups of the acting group $G$. Note that Theorem \ref{theo-LQA} does not require the almost normal subgroup $H$ to be finite.

\begin{proof}
   We show that (1) holds if and only if (3) holds. By Theorem \ref{converse}, the space $X$ is partitioned into a finite number of clopen sets $X_i:=\overline{X_H(g_i)}$, where $g_1 = 1_G,g_2,\ldots, g_n$ are representatives of the cosets of $N_G(H)$. 
   Fix $\epsilon < {\rm min}\{{\rm dist}(X_i,X_j):  1 \leq i \ne j \leq n\}$. Then an open set $U$ of diameter less than $\epsilon$ is contained in $X_i$ for some $1 \leq i \leq n$. Let $g \in G$ and $V \subset U$ be such that $g|V = {\rm id}$, then for any $x \in V$ we have $g \in G_x$. Since the points in $X_H(g_i)$ have equal stabilizers, we get $g \in G_x$ for $x \in U \cap X_H(g_i)$. Since the latter is dense in $U$, we obtain $g \in G_x$ for all $x \in U$, and so $g|U = {\rm id}$. This shows that the action $(X,G)$ is locally quasi-analytic. 

    For the converse, suppose that $(X,G)$ is locally quasi-analytic, and let $\epsilon >0$ be such that for any open $U \subset X$ with $\diam(U) <\epsilon$, and any $g \in G$, if $g|V = {\rm id}$ for some open $V \subset U$, then $g|U = {\rm id}$. Fix such $U$. Recall that if $x \in U$ is a point without holonomy, and $g \in G_x$, then $g$ fixes an open neighborhood of $x$ in $U$, and therefore must fix every point in $U$. It follows that all points without holonomy in $U$ have equal stabilizers. Fix one such point without holonomy $x \in U$ and set $H: = G_x$. Next, consider a covering $\mathcal U = \{gU: g \in G\}$ of $X$ by open sets, and choose a finite subcover $\mathcal U' = \{g_iU: 1 \leq i \leq n\}$. Since the orbit of $x$ is dense in $X$, all points without holonomy in $g_i U$ have stabilizers equal to $g_i H g_i^{-1}$, and if $g_iU \cap g_jU \ne \emptyset$, then $g_i H g_i^{-1} = g_jH g_j^{-1}$ by Lemma \ref{auxiliar}. Since $\mathcal U'$ is finite, $H$ is almost normal. 
    
    The equivalence between (2) and (1) follows from the definition of the stabilizer URS.\end{proof}

\begin{example}
{\rm  In Proposition \ref{equivalence-URS}, statement (1), the finite stabilizer of $x \in X$ need not be an almost normal subgroup. Consider the infinite dihedral group, i.e.
$$G = \langle a,b \mid bab=a^{-1},b^2=1_G \rangle,$$
and let $G_i = \langle a^{2^i},b \rangle$ for $i\geq 0$. Then $\mathcal G: G = G_0 \supset G_1 \supset \cdots$ is an infinite decreasing sequence of finite index subgroups of $G$. Consider the $G$-odometer $(X,G)$ defined by the group chain $\mathcal G$ (see Section \ref{odom-action}) and let $v_i = 1_GG_i$. Then $(v_0,v_1,\ldots)$ is a point in $X$ with stabilizer $K = \bigcap_{i \geq 0} G_i =\langle b \rangle$. The subgroup $K$ has an infinite number of conjugates, see \cite[Example 2.8]{GL2021},  and so by Proposition \ref{equivalence-URS} it cannot be an URS. It turns out that $(X,G)$ is topologically free, i.e.~$H = \{1_G\}$ is an URS for this action.} 
\end{example}

\begin{proof}[Proof of Theorem \ref{theo-1}] The equivalence of statements (1)-(3) is given by Proposition \ref{equivalence-URS}, the equivalence of statements (1)-(3) to (4) is given by Theorem \ref{theo-LQA} plus the assumption that there is a point with finite stabilizer.
\end{proof}

\subsection{Holonomy, invariant measures,  and almost normal stabilizers.}

In this section we prove Theorem \ref{coro-IRS}. 

The proof of the next lemma is straightforward and we include it here for completeness.

\begin{lemma}\label{finite-conjugacy-classes}
Let $(X,G,\mu)$ be a  dynamical system  equipped with an invariant measure $\mu$. If  $H$ is a subgroup of $G$ such that $\mu(X_H)>0$,   then  $H$ is almost normal.
\end{lemma}
\begin{proof}

 Let  $\Stab: X \to S(G)$ be the map introduced in Section \ref{sec-inv-random}, and let $\nu=\Stab_*\mu$. If  $D\subseteq G$ is a set containing exactly one representing element of each class in $\{gN_G(H): g\in G\}$, then 
$\{G_x: x\in X_H\}=\bigcup_{g\in D}gHg^{-1}$. Since this union is disjoint, 
we have
$$
0<\mu(X_H)\leq \nu(\{G_x: x\in X_H\})=\nu(\{gHg^{-1}: g\in D\})=|D|\nu(\{H\}),
$$
which implies that $|D|=[G:N_G(H)]<\infty.$
\end{proof}

The next proposition shows that a.e.~stabilizer of the dynamical system $(X,G)$ is finite if and only if they are conjugate subgroups a.e.

\begin{prop}\label{old-corollary-00}
Let $(X,G,\mu)$ be a minimal dynamical system equipped with an ergodic measure $\mu$. The following statements are equivalent:
\begin{enumerate}
\item   $|G_x|<\infty$, for $\mu$-a.e.~$x\in X$. 

\item There exists  an almost normal finite subgroup $H$ of $G$ such that $\mu(X_H)=1$.

\item There exists a finite subgroup $H$ of $G$ with $\Stab_*\mu(\{H\})>0$, i.e.~${\rm Stab}_*\mu$ is atomic.
\end{enumerate}
\end{prop}
\begin{proof}
 Assume (1), and observe that for every $n\geq 1$, the set $\{ x\in Y: |G_x|=n\}$ is Borel and invariant. Thus, since 
 $$\{x\in X: |G_x|<\infty  \}=\bigcup_{n\geq 1}\{ x\in X: |G_x|=n\},$$
 the ergodicity of $\mu$ implies there exists a unique $n\geq 1$ such that $$\mu(\{ x\in X: |G_x|=n\})=1.$$   Since the collection of finite subsets of a countable set is countable (we can think of the subsets of $G$ with cardinality $k$ as elements of $G^k$), we can write the set of all $x\in X$ such that $|G_x|\leq n$, as a disjoint union of sets $\{x\in X: G_x\in  \Conj(H_i)\}$, where $\{H_i: i\in I\}$ is a countable (or finite) collection of  subgroups of $G$ with cardinality less than or equal to $n$.
 Furthermore, since for every $i\in I$ the set $ \{x\in X: G_x\in \Conj(H_i)\}$ is invariant and Borel, the ergodicity of $\mu$ implies there exists $i\in I$ such that $\mu(X_{H_i})=1$, and by   Lemma \ref{finite-conjugacy-classes} $H_i$ is almost normal. Thus (2) holds if and only if (1) holds.

Assume (2), so there exists a finite subgroup $H$ of $G$ such that $\mu(X_H)=1$. Then Lemma \ref{finite-conjugacy-classes} implies that $H$ is almost normal. Thus, the orbit of $H$ in $\Stab(X)$ is finite, and because $\Stab^{-1}\{\Conj(H)\}=X_H$, we deduce that  $\Stab_*\mu(\Conj(H))=1$, which implies that $\Stab_*\mu(\{H\})>0$, i.e.~(3). 

Conversely, if there exists a finite group $H$ of $G$ such that $\Stab_*\mu(\{H\})>0$, then $H$ is almost normal, because $\Stab_*\mu$ is invariant by the action by conjugation. This also implies that $\mu(X_H)>0$. Since $\mu$ is ergodic and $X_H$ is invariant, we get $\mu(X_H)=1$.  
\end{proof}

We now can summarize the results of Propositions \ref{equivalence-URS} and \ref{old-corollary-00} into Theorem \ref{finite-summarize}, highlighting dependencies between generic topological and measure-theoretical properties of dynamical systems with almost normal stabilizers.

\begin{theo}\label{finite-summarize}
Let $(X,G,\mu)$ be a minimal dynamical system equipped with an ergodic measure $\mu$. If $K$ is a finite subgroup of $G$ such that $X_K\neq\emptyset$, then there exists a finite subgroup $H$ of $G$, such that (1)$\iff$ (2) $\implies$ (3) $\iff$ (4) $\iff$ (5) $\iff$ (6), where

\begin{enumerate}
\item    $\mu(X_H)=1$.  
\item   $\Stab_*\mu(\{H\})>0$ and $\Conj(H)$ is the IRS of $(X,G)$.  
\item  There exists $x\in X_H$ with trivial holonomy. 
\item    $X_H$ coincides with the set of points of $X$ with trivial holonomy. 
\item   $H$ is almost normal. 
\item  $\Conj(H)$ is the URS of $(X,G)$.
\end{enumerate}
\end{theo}

Note that in Theorem \ref{finite-summarize} (1) is equivalent to (2), (3) - (6) are mutually equivalent. Either (1) or (2) implies (3) - (6) but the converse is not true.
This is not unexpected. Indeed, every minimal essentially free action is topologically free, while the converse does not hold. Similarly, an IRS of a dynamical system need not be supported on its stabilizer URS. Below in Theorem \ref{thm-irs-urs}, which is a restatement of Theorem \ref{coro-IRS}, we give a condition under which the IRS is supported on the IRS. Definitions of an allosteric group and of commensurable groups are recalled in the introduction just above Theorem \ref{coro-IRS}.

\begin{theo}\label{thm-irs-urs}
Let $(X,G,\mu)$ be a minimal dynamical system equipped with an ergodic measure $\mu$. Let $H$ be a finite subgroup of $G$, and suppose $(X,G,\mu)$ and $H$ satisfy the conditions (3)-(6) of Theorem \ref{finite-summarize}.
If $G$ is not allosteric and $N_G(H)/H$ is commensurable to $N_G(H)$, then all sentences in Theorem \ref{finite-summarize} are equivalent. In particular, the IRS of $(X,G,\mu)$ is supported on the stabilizer URS, and so the action $(X,G)$ has no essential holonomy.
\end{theo}

\begin{proof}
We will prove that (5) in Theorem \ref{finite-summarize} implies (1).
For that, consider the action of $N_G(H)/H$ on $\overline{X_H(1_G)}$, i.e.~on the closure of the set of points with stabilizer equal to $H$, given by $(gH)x=gx$, for every $x\in \overline{X_H(1_G)}$ and $g\in N_G(H)$. This action is well-defined and minimal, and its set of points with trivial stabilizer is $X_H(1_G)$. Furthermore, the measure $[G:N_G(H)]\mu$ is an ergodic measure of $(\overline{X_H(1_G)}, N_G(H)/H)$. On the other hand, since $N_G(H)$ has finite index in $G$, if $G$ is not allosteric, then \cite[Theorem 2.9]{Joseph2024} implies that $N_G(H)$ is not allosteric. By assumption $N_G(H)/H$ is commensurable to $N_G(H)$, so again by \cite[Theorem 2.9]{Joseph2024}, the group $N_G(H)/H$ is not allosteric. This implies that its measure is supported on the set of points with trivial stabilizers, i.e.~$\mu(X_H(1_G))=\frac{1}{[G:N_G(H)]}$. Then $\mu(X_H)=1$. The rest follows from Theorem \ref{finite-summarize}. 
\end{proof}

 \begin{remark}\label{remark-allosteric}
{\rm In Theorem \ref{finite-summarize}, (5) $\implies$ (1)  does not always hold. Indeed, if $G$ is allosteric and $H$ is a normal finite subgroup of $G$, then   $G/H$ is allosteric (see Corollary \ref{allosetric-sufficient} below). Then, there exists a minimal system $(X, G/H)$ equipped with an ergodic measure $\mu$, such that  the set $X_{\{1_{G/H}\}}\neq \emptyset$ and $\mu\big(X_{\{1_{G/H}\}}\big)=0$. The induced action of $G$ on $X$ satisfies $X_H=X_{\{1_{G/H}\}}\neq \emptyset$  and $\mu(X_H)=0$. 
}
\end{remark}

\begin{coro}\label{finite-summarize-RF}
Let $(X,G)$ be a minimal dynamical system equipped with an ergodic measure $\mu$. Suppose that $G$ is not allosteric and residually finite. Then all the statements of Theorem \ref{finite-summarize} are equivalent.
 \end{coro}
\begin{proof}
Let $H$ be an almost normal finite group of $G$. Since $G$ is residually finite, there exists a finite index subgroup $\Gamma$   of $G$ such that $H\cap \Gamma=\{1_G\}$. This implies that $(\Gamma\cap N_G(H))/H$ and $\Gamma\cap N_G(H)$ are isomorphic. Since $(\Gamma\cap N_G(H))/H$ and $\Gamma\cap N_G(H)$ are finite index subgroups of $N_G(H)/H$ and $N_G(H)$, respectively, then the groups are commensurable. 
\end{proof}

\section{Finite-to-one factor maps and almost normal stabilizers}\label{sec-factors}

In this section we prove Theorem \ref{characterization-factor}. The bulk of the proof is in Proposition \ref{factor-stabilizer}. We start with a few technical results, some of them consequences of Theorem \ref{theo-1}. The proofs are straightforward, but we include them here for completeness.

\subsection{Technical results}

Let $m\in \NN$. We denote by $S_m$ the symmetric group on $m$ elements. A subgroup $H$ of $S_m$ is {\it regular} if for every $\sigma\in H\setminus \{1_G\}$ and every $i\in \Sigma_m = \{1,\cdots, m\}$, we have $\sigma(i)\neq i$, i.e.~every $\sigma \in H$ acts on $\Sigma_m$ without fixed points.
 

\begin{lemma}\label{finite-stabilizer}
 Suppose that $G$ acts on the non-empty sets $X$ and $Y$. Let $\pi:X\to Y$ be a  map that commutes with the actions. If $y\in Y$ is such that $|\pi^{-1}\{y\}|=m<\infty$, 
then $[G_y:G_x]\leq m!$, for every $x\in \pi^{-1}\{y\}$. In particular, if there exists $x\in \pi^{-1}\{y\}$ with trivial stabilizer, then   $G_y$ is isomorphic to a regular subgroup of $S_m$.
\end{lemma}
\begin{proof}
Let $y\in Y$ and $\pi^{-1}\{y\}=\{x_1,\ldots, x_m\}$.  We identify $\pi^{-1}\{y\}$  with the set $\Sigma_m$. 
Then each $g \in G_y$ defines a permutation $\sigma_g: \Sigma_m \to \Sigma_m $, where $\sigma_g(i) \in \Sigma_m$ is such that
$g x_i=x_{\sigma_g(i)}$, and we get a group homomorphism $\phi: G_y\to S_m$ given by $\phi(g)=\sigma_g$, for every $g\in G_y$.
 If $\sigma_g=id$, then $gx_i=x_i$ for every $1\leq i\leq m$, which implies that $g\in \bigcap_{j=1}^m G_{x_j}$. Thus,  $\ker(\phi)$ (the kernel of $\phi$) is a subgroup of $G_x$, for every $x\in \pi^{-1}\{y\}$. Since $G_y/\ker(\phi)$ is isomorphic to a subgroup of $S_m$, we have that $\ker(\phi)$ is a subgroup of $G_y$ of index at most $m!$. Finally, due to $\ker(\phi)\subseteq G_x\subseteq G_y$, we deduce that $[G_y:G_x]\leq m!$, for every $x\in \pi^{-1}\{y\}$. 
   \end{proof}

\begin{remark}{\rm 
Lemma \ref{finite-stabilizer} implies that there are some restrictions on the acting group $G$ in order to have finite-to-one factor maps from free dynamical systems to non-free dynamical systems. For example, groups without torsion elements do not admit such a factor map.}
\end{remark}

The following is a classic result in ergodic theory. See for example \cite[Lemma 3.1]{Yoo2018}.

\begin{lemma}\label{constant-pre-images}(\cite[Lemma 3.1]{Yoo2018})
Let $\pi: X\to Y$ be a factor map between  the dynamical systems $(X,G)$ and $(Y,G)$.  Suppose that $(Y,G)$ has an ergodic measure $\mu$ invariant under $G$. Then there exists $n\in \NN\cup\{\infty\}$ such that
$$
\mu(\{y\in Y: |\pi^{-1}\{y\}|=n\})=1.
$$
\end{lemma}

The next proposition is an extension of Theorem \ref{theo-1} and Theorem \ref{coro-IRS} to the case of finite-to-one factor maps of minimal dynamical systems. It shows, in particular, that the property of having no essential holonomy is preserved under finite factors.
 
\begin{prop}\label{factor-topologically-free}
 Let $(X,G)$ be a topologically free minimal system.   If $(Y,G)$ is a dynamical system, and   $\pi:X\to Y$ is a factor map such that 
 $$
 \pi^{-1}\{y\in Y: |\pi^{-1}\{y\}|<\infty\}\cap \{x\in X: G_x=\{1_G\}\}\neq\emptyset,
 $$
then there exists an almost normal finite subgroup $H$ of $G$ such that
 $$
 \pi^{-1}Y_H\subseteq \{x\in X: G_x=\{1_G\}\}.
 $$
 In particular, if $Y$ is a metric space, then $(Y,G)$ is locally quasi-analytic.
If, in addition, there is an ergodic measure $\mu$, such that $(X,G,\mu)$ is an essentially free system, and
$$
\mu(\pi^{-1}\{y\in Y: |\pi^{-1}\{y\}|<\infty\})=1,
$$
then the following is true:
\begin{enumerate}
    \item There exists an almost normal finite subgroup $H$ of $G$ such that $\pi_*\mu(Y_H)=1$.
    \item The IRS of the action $(Y,G,\pi_* \mu)$ is supported on the stabilizer URS. In particular, $(Y,G,\pi_* \mu)$ has no essential holonomy. 
\end{enumerate}
\end{prop}
\begin{proof}
Since there exists $y\in Y$ with a finite fiber that contains a point with trivial stabilizer, Lemma \ref{finite-stabilizer} implies that $G_y$ is finite. So, by Theorem \ref{theo-1}, there exists an almost normal finite group $H$ of $G$ such that $Y_H\neq\emptyset$.  Suppose that for some $z\in Y_H$ there is some $x\in \pi^{-1}\{z\}$ such that $G_x\neq \{1_G\}$. Since $G_x$ must be finite, $G_x$ cannot be almost normal, otherwise, the orbit of $G_x$ in $\overline{\Stab(X)}$ would be a minimal set, different from $\{1_G\}$, the unique URS of $(X,G)$. Furthermore, we have $gG_xg^{-1}\subseteq \Stab(g\pi(x))=gHg^{-1}$, for every $g\in G$. This implies  $gG_xg^{-1}\subseteq H$, for every $g\in N_G(H)$.  From this we get that there are infinitely many different conjugate $g_1G_xg_1^{-1}, g_2G_xg_2^{-1},\ldots$ in $H$, which contradicts that $H$ is finite. This shows that $\pi^{-1}Y_H\subseteq \{x\in X: G_x=\{1_G\}\}.$

Now suppose $(X,G,\mu)$ is essentially free. From Lemma \ref{constant-pre-images}, there exists finite $n\geq 1$ such that for $\pi_*\mu$-a.e.~$y\in Y$ we have $|\pi^{-1}\{y\}|=n$.
Then, applying Lemma \ref{finite-stabilizer}, we get $|G_y|\leq n$, for $\pi_*\mu$-a.e.~$y\in Y$.  The rest of the proof follows from Proposition \ref{old-corollary-00}.
\end{proof}

\begin{remark}
{\rm From  \cite[Theorem 3.3]{Yoo2018}, it is possible to deduce that if $\pi:(X,G)\to (Y,G)$ is a factor map which is finite-to-one $\nu$-a.e., where $\nu$ is an ergodic measure of $(Y,G)$, then there exists   a lift $\mu$ of $\nu$ on $(X,G)$. Thus, the second half of Proposition \ref{factor-topologically-free} can be stated as follows: Let $(X,G)$ and $(Y,G)$ be dynamical systems.  Suppose $(Y,G)$ is equipped with an ergodic measure $\nu$, and that there exists a factor map $\pi:X\to Y$,  finite-to-one $\nu$-a.e. If $(X,G,\mu)$ is essentially free with respect to a lift $\mu$ of $\nu$, then there exists an almost normal finite subgroup $H$ of $G$ such that $\nu(Y_H)=1$, thus the IRS of the system $(Y,G,\nu)$ is supported on its stabilizer URS, and the action $(Y,G,\nu)$ has no essential holonomy. If $Y$ is a metric space, then $(Y,G,\nu)$ is also locally quasi-analytic.}    
\end{remark}

\subsection{Factors of actions with finite almost normal stabilizers} 

\begin{prop}\label{factor-stabilizer} Let $H$ be  an almost normal finite subgroup of $G$. Let $(X,G)$ be a     dynamical system  such that there exists a factor map $\phi:X\to G/N_G(H)$. 
 Then there exist a   dynamical system $(Y,G)$,  and a factor map $\pi:X\to Y$, such that:
  \begin{enumerate}
  \item There is a factor map $\tilde{\phi}: Y\to G/N_G(H)$ such that $\tilde{\phi}\circ\pi=\phi$.
  \item $G_{y}= gHg^{-1}G_x$, for every $y\in \tilde{\phi}^{-1}\{gN_G(H)\}$, and any $x\in \pi^{-1}\{y\}$.
  \item $|\pi^{-1}\{y\}|=\frac{|H|}{|gHg^{-1}\cap G_x|}\leq |H|$, for every $y\in \tilde{\phi}^{-1}\{gN_G(H)\}$, and any $x\in \pi^{-1}\{y\}$.
  \end{enumerate}
  Furthermore, if $(Y',G)$ is another dynamical system for which there exists a factor map $\pi':X\to Y'$, such that
\begin{itemize}
\item[(i)] there exists a factor map $\tilde{\phi}':Y'\to G/N_G(H)$ such that $\tilde{\phi}'\circ\pi'=\phi$,
\item[(ii)] $gHg^{-1}\subseteq G_y$, for every $y\in \tilde{\phi}'^{-1}\{gN_G(H)\}$,
\end{itemize}
then there is a factor map $\varphi:Y\to Y'$ such that $\varphi\circ\pi=\pi'$.
\end{prop}
\begin{proof}
Let $g_1,\ldots, g_k\in G$ be representing elements of each class in $G/N_G(H)$. For every $1\leq i\leq k$, we set $X_i=\phi^{-1}\{g_iN_G(H)\}$. 
Let us define the following relation on $X$: $x_1\sim x_2$ if and only if there exist $1\leq i\leq k$ and  $h\in H$  such that $x_1,x_2\in X_i$ and $x_1=g_ihg_i^{-1}x_2$. This is an equivalence relation, with equivalence classes contained in $X_i$, $1 \leq i \leq k$.

\emph{Claim 2}. The set $R=\{(x_1,x_2): x_1\sim x_2, x_1,x_2\in X\}$ is $G$-invariant and closed in $X\times X$. 

\begin{proof}[Proof of Claim 2]
Indeed, if $(x_1,x_2)\in R$ then there exist $1\leq i\leq k$ and $h\in H$ such that $x_1=g_ihg_i^{-1}x_2$. Let $p\in G$, $1\leq j\leq k$ and $l\in N_G(H)$ such that $pg_i=g_jl$.  Then 
$$px_1= pg_ihg_i^{-1}x_2=g_jlhg_i^{-1}x_2=g_j(lhl^{-1})lg_i^{-1}x_2=g_jh'g_j^{-1}px_2,$$
where $lhl^{-1}=h'\in H$. This shows that $(px_1,px_2)\in R$.  On the other hand, suppose that $(x,y)\in (X\times X) \setminus R$. If $x\in X_i$ and $y\in X_j$ with $i\neq j$, since $X_i$ and $X_j$ are clopen, there exists neighborhoods $V_x\subseteq X_i$ and $V_y\subseteq X_j$ of $x$ and $y$, respectively. We have that $V_x\times V_y$ is a neighborhood of $(x,y)$ contained in  $(X\times X)\setminus R$. If $x$ and $y$ are in the same $X_i$, since $X$ is Hausdorff, there exist an open neighborhood $V$ of $x$, and open neighborhoods $V_h$ of $g_ihg_i^{-1}y$, such that $V\cap V_h=\emptyset$, for every $h\in H$. Since $H$ is finite, the continuity of the action implies there exists an open neighborhood $U$ of $y$ such that if $z\in U$, then $g_ihg_i^{-1}z\in V_h$, for every $h\in H$. The set $V\times U$ is an open neighborhood of $(x,y)$ contained in $(X\times X)\setminus R$. This shows that $R$ is closed. 
\end{proof}

Since $R$ is a closed $G$-invariant equivalence relation, from \cite[p.~23]{Auslander1988} we get that $Y=X/\sim$ is compact and Hausdorff, and the projection  $\pi: X\to Y$ is a factor map, where the action of $G$ on $Y$ is given by $g\pi(x)=\pi(gx)$, for every $x\in X$. The definition of $\sim$ implies that $\tilde{\phi}: Y\to G/N_G(H)$ given by $\tilde{\phi}(y)=\phi(x)$, where $x$ is any element in $\pi^{-1}\{y\}$, is a well defined factor map that verifies (1). 

Let $y\in Y$ and $g\in G$ be such that $y\in \tilde{\phi}^{-1}\{gN_G(H)\}$. This yields $\pi^{-1}\{y\}\subseteq \phi^{-1}\{gN_G(H)\}$. From this we get that $G_x$ is a subgroup of $gN_G(H)g^{-1}$, for every $x\in \pi^{-1}\{y\}$. Since $gHg^{-1}$ is normal in $gN_G(H)g^{-1}$, we obtain that $gHg^{-1}G_x$ is a subgroup of $gN_G(H)g^{-1}$, and so of $G$, for every $x\in \pi^{-1}\{y\}$. 
If $l\in G_y$, then for $x\in \pi^{-1}\{y\}$, we have $\pi(lx)=\pi(x)$. This implies that $lx\sim x$, from which we deduce  there exists $h\in H$ such that $ghg^{-1}lx=x$. From this follows that $ghg^{-1}l\in G_x$, and then $l \in gHg^{-1}G_x$. This shows that $G_y\subseteq gHg^{-1}G_x$.
Conversely, if $h\in H$ and $l\in G_x$, then $ghg^{-1}ly=ghg^{-1}l\pi(x)=ghg^{-1}\pi(lx)=ghg^{-1}\pi(x)=\pi(ghg^{-1}x)=\pi(x)=y$. This shows that $G_{y}=gHg^{-1}G_x$, and so   
$$gHg^{-1}G_x=G_y =gHg^{-1}G_{x'}, \mbox{ for every } x,x'\in \pi^{-1}\{y\}.$$

From the definition of $\sim$, the cardinality of $\pi^{-1}\{y\}$, is equal to
$$|gHg^{-1}/(gHg^{-1}\cap G_x)|=\frac{|H|}{|gHg^{-1}\cap G_x|}, \mbox{ for every } x\in \pi^{-1}\{y\}.$$
Thus, statements (2) and (3) have been established.

Let $(Y',G)$ be a dynamical system such that there exists a factor map $\pi':X\to Y'$, such that (i) and (ii) in the hypothesis of the proposition are satisfied. 
If $(x_1,x_2)\in R\subseteq X\times X$, then there exists $g\in G$  such that $x_1$ and $x_2$ are in $\phi^{-1}\{gN_G(H)\}$, and $x_1=ghg^{-1}x_2$, for some $h\in H$. Then by (i) we have $\pi'(x_1), \pi'(x_2)\in \tilde{\phi}'^{-1}\{gN_G(H)\}$, and   $\pi'(x_1)=ghg^{-1}\pi'(x_2)=\pi'(x_2)$,  thanks to point (ii). This shows that 
$$R \subseteq  R'= \{(x_1,x_2)\in X\times X: \pi'(x_1)=\pi'(x_2)\},$$ which implies there exists a factor map $\varphi:Y\to Y'$ such that $\pi'=\varphi\circ\pi$.
 \end{proof}

\begin{remark}\label{remark_allosteric}
{\rm  Observe that if $H$ is a finite normal subgroup of $G$, then $H$ is finite almost normal, and since every   system $(X, G)$ is an extension of the trivial group $G/N_G(H)$, then every  system $(X,G)$ satisfies the hypothesis of Proposition \ref{factor-stabilizer} with respect to $H$.}    
\end{remark}

\begin{lemma}\label{key}
 If $(X,G)$ is a minimal topologically free dynamical system, and $H$ is an almost normal finite subgroup of $G$, then $G_x\cap gHg^{-1}=\{1_G\}$, for every $x\in X$ and $g\in G$.     
\end{lemma}

\begin{proof}
Suppose there exist $x\in X$ and $g\in G$,  such that there exists  $l\in (G_x\cap gHg^{-1})\setminus\{1_G\}$. This implies that the stabilizer of every point in the orbit of $x$ contains a conjugate of $l$, i.e.~a non-trivial element in $C=\bigcup_{g\in G}gHg^{-1}$. Since $(X,G)$ is topologically free, there exists $y\in X$ with trivial stabilizer, and the minimality of $(X,G)$ implies that for every neighborhood $U$ of $y$ there is some $g_U\in G$ such that $g_Ux\in U$. Since $C$ is finite, this implies there exists $g\in G$ such that $glg^{-1}\in G_{g_Ux}$, for every neighborhood $U$ of $x$. From the continuity of the action follows that $glg^{-1}y=y$, which is a contradiction. 
\end{proof}

We now prove Theorem \ref{characterization-factor}. We restate it here for the convenience of the reader.

\begin{theo}\label{characterization-factor-1}
Let $(X,G)$ be a minimal dynamical system, and let $H$ be a finite subgroup of $G$. Then,
\begin{enumerate}
\item If $(X,G)$ is topologically free, and $H$ is  almost normal,  then the following sentences are equivalent:

\medskip

\begin{itemize}
\item[(1.1)] There exists a factor map $\phi:X\to G/N_G(H)$.

\item[(1.2)] There exist a minimal dynamical system $(Y,G)$,  and a $|H|$-to-one factor map $\pi:X\to Y$ such that 
$$
\pi^{-1}Y_H=\{x\in X: G_x=\{1_G\}\},
$$
and ${\rm Conj}(H)$ is the stabilizer URS of $(Y,G)$.
\end{itemize}

\item If $(X,G,\mu)$ is essentially free with respect to an ergodic measure $\mu$,  then the following sentences are equivalent:

\medskip

\begin{itemize}
\item[(2.1)] $H$ is almost normal and there exists a factor map $\phi:X\to G/N_G(H)$.

\item[(2.2)] There exist a minimal dynamical system $(Y,G)$,  and a $|H|$-to-one factor map $\pi:X\to Y$ such that 
$$
\pi^{-1}Y_H=\{x\in X: G_x=\{1_G\}\}. $$

\item[(2.3)] There exist a minimal  dynamical system $(Y,G,\pi_*\mu)$ and a $|H|$-to-one $\pi_*\mu$-a.e.~factor map $\pi:X\to Y$ such that $\pi_*\mu(Y_H)=1$. 
\end{itemize}
\end{enumerate}
 \end{theo}

\begin{remark}
{\rm 
    If (1.2) or, equivalently, (2.2) of Theorem \ref{characterization-factor-1} hold, then ${\rm Conj}(H)$ is the stabilizer URS of $(Y,G)$ by Theorem \ref{theo-1}, (3). If, in addition, $Y$ is a metric space, then $(Y,G)$ is locally quasi-analytic by Theorem \ref{theo-LQA1}.

    If (1.3) holds, then ${\rm Conj}(H)$ is the stabilizer URS of $(Y,G,\pi_*\mu)$ by the argument in the previous paragraph, and its IRS is supported on its stabilizer URS. In particular, $(Y,G,\pi_*\mu)$ has no essential holonomy.
    }
\end{remark}
 
\begin{proof}[Proof of Theorem \ref{characterization-factor-1}]
Suppose that $(X,G)$ is minimal and topologically free, and $H$ is almost normal. By Lemma \ref{key}, we have $G_x\cap gHg^{-1}=\{1_G\}$, for every $x\in X$. Then, if we assume (1.1),  
Proposition \ref{factor-stabilizer} implies there exists a minimal dynamical system $(Y,G)$, and a 
finite-to-one factor map $\pi:X\to Y$ such that for every $g\in G$,
\begin{eqnarray*}
\pi^{-1}\{y\in Y: G_y=gHg^{-1}\} & = &\{x\in X: G_x\subseteq gHg^{-1}\}\\
                                    & = &\{x\in X: G_x=\{1_G\}\}.
\end{eqnarray*}
Thus, we have
$$
\pi^{-1}Y_H=\{x\in X: G_x=\{1_G\}\}.
$$
 We claim that $\pi$ is $|H|$-to-$1$. Indeed, by Proposition \ref{factor-stabilizer} $|\pi^{-1}(\pi(x))| = \frac{|H|}{|gHg^{-1} \cap G_{\pi(x)}|}$, for an appropriate $g \in G$, so this is clear for $x \in X$ with $G_x = \{1_G\}$. Suppose $x \in X$ is such that $G_x \ne \{1_G\}$, and let $K = gHg^{-1} \cap G_x$ for some $g \in G$. Since $\bigcup_{g \in G}gHg^{-1}$ is a finite set, $\bigcup_{g \in G} gKg^{-1}$ is a finite set, and so $K$ has a finite number of conjugates. The orbit of $x$ is dense in $X$, and so accumulates at a point $y$ with $G_y = \{1_G\}$. Since $K$ has a finite number of conjugates, we may choose a sequence of points $\{x_i\}_{i\geq 1} \subset O(x)$, converging to $y$, such that $G_{x_i} \supset gKg^{-1}$ for a fixed $g \in G$. Then by continuity of the action $gKg^{-1} \subset G_y =\{1_G\}$, which implies that $K = \{1_G\}$, and $\pi$ is $|H|$-to-$1$.

Assume that we have (1.2). Since $(Y,G)$ is minimal, $H$ is almost normal, and $Y_H\neq \emptyset$, Corollary \ref{converse-minimal} implies there exists a factor map $\varphi: Y\to G/N_G(H)$. The composition $\phi=\varphi\circ \pi$ is the map described in (1.1). 

Suppose that $(X,G,\mu)$ is minimal and essentially free  with respect to an ergodic measure $\mu$. If (2.1) holds, and since essentially free implies topologically free, we get (2.2).  If (2.2) holds, then it is direct that (2.3) holds. If we assume (2.3), then Lemma \ref{finite-conjugacy-classes} implies that $H$ is almost normal. Then, by Corollary \ref{converse-minimal}, there exists a factor map $\varphi: Y\to G/N_G(H)$. The composition $\phi=\varphi\circ \pi$ is the map in (2.1). 
\end{proof}

Recall that a group $G$ is allosteric if there exists a minimal topologically free action of $G$ on a compact Hausdorff space which is not essentially free w.r.t.~an ergodic measure. The next corollary shows that the property of being allosteric is preserved under quotients by finite normal subgroups.

\begin{coro}\label{allosetric-sufficient} The following sentences are equivalent:
\begin{enumerate}
\item The group $G$ is allosteric  
\item For every finite normal subgroup $H$ of $G$, the group  $G/H$ is allosteric.  \end{enumerate}
\end{coro}
\begin{proof}
Observe that (2) implies (1) because it is enough to take $H=\{1_G\}$.

Suppose that $G$ is allosteric. Then there exists a minimal system $(X,G,\mu)$ equipped with an ergodic measure $\mu$, such that $(X,G,\mu)$ is topologically free but not essentially free with respect to $\mu$. By Remark \ref{remark_allosteric}, we can apply Theorem \ref{characterization-factor} to $(X,G,\mu)$ and $H$. Thus, there exists a minimal system $(Y,G)$ and a $|H|$-to-one factor map $\pi:X\to Y$ such that  
$$
\pi^{-1}Y_H=\{x\in X: G_x=\{1_G\}\},
$$
which implies
$$
\pi_*\mu(Y_H)=\mu(\{x\in X: G_x=\{1_G\}\})=0.
$$
Note that $\pi_*\mu$ is an ergodic measure for the minimal dynamical system $(Y, G/H)$ given by
$$
gHy=gy \mbox{ for every } y\in Y, g\in G.
$$
Since
$$
\{y\in Y: (G/H)_{y}=\{1_{G/H}\}\}=\{y\in Y: G_y=H\},
$$
we deduce that $(Y,G/H,\pi_*\mu)$ is topologically free but not essentially free w.r.t.~$\pi_*\mu$. 
\end{proof}

 Recall that a dynamical system $(X,G,\mu)$ is \emph{weakly mixing}, if every non-empty open invariant subset $U \subset X \times X$, invariant with respect to the product action of $G$ on $X \times X$, is dense in $X \times X$ (see \cite[p.~131]{Auslander1988}). 

\begin{coro}\label{cor-weakly-mixing}
 Let $(X,G,\mu)$ be  a minimal dynamical system that is essentially free with res\-pect to some ergodic measure $\mu$.  If $(X,G,\mu)$ is weakly mixing, then every effective system $(Y,G,\pi_*\mu)$ for which there exists a finite-to-one $\pi_*\mu$-a.e.~factor map $\pi:X\to Y$,  is essentially free with respect to $\pi_*\mu$.
\end{coro}
\begin{proof}
 If $(Y,G)$ is a system such that  there exists a factor map $\pi:X\to Y$, then $(Y,G)$ is minimal. If in addition, we have $|\pi^{-1}\{y\}|<\infty$ for $\pi_*\mu$-a.e.~$y\in Y$, then Proposition \ref{factor-topologically-free} implies  there exists an almost normal finite subgroup $H$ of $G$ such that   $\pi_*\mu(Y_H)=1$. From Theorem \ref{characterization-factor} we get that the equicontinuous system given by the left multiplication action of $G$  on $G/N_G(H)$ is a factor of $(X,G)$. If $(X,G,\mu)$ is weakly mixing, then $G/N_G(H)$ is trivial (see \cite[Theorem 13, p.~133]{Auslander1988}), which implies that $H$ is normal and then it is in the stabilizer of every element of $Y$. Thus, if the action of $G$ on $Y$ is effective, then $H=\{1_G\}$. This shows that $(Y,G,\pi_*\mu)$ is essentially free.   \end{proof}

\subsection{Examples}

Let $G$ be a countable residually finite group. A decreasing sequence $(\Gamma_n)_{n\in\NN}$ of finite index subgroups of $G$ is called {\it Farber} (see \cite{KKN17}) if the $G$-odometer (see Section \ref{odom-action})  associated to $(\Gamma_n)_{n\in\NN}$ is essentially free with respect to its unique invariant probability measure $\mu$. For example, a decreasing sequence of normal finite index subgroups of $G$ with trivial intersection is Farber.

The next example is a consequence of Theorem \ref{characterization-factor}. 
\begin{example}\label{Odometer}
{\rm Let $G$  be a  residually finite group, and  let $(\Gamma_n)_{n\in\NN}$ be a  Farber sequence of subgroups of $G$.  Let $H$ be an almost normal finite subgroup of $G$. Let $(X,G)$ be the $G$-odometer associated to $(\Gamma_n)_{n\in\NN}$, and  let $\mu$ be  its unique ergodic measure. The following statements are equivalent:
\begin{enumerate}
\item There exists $n\in\NN$ such that $\Gamma_n\subseteq N_G(H)$.

\item There exist a minimal dynamical system $(Y,G)$, and a $|H|$-to-one factor map $\pi:X\to Y$, such that $\pi_*\mu(Y_H)=1$. 

\item There exist a minimal dynamical system $(Y,G)$,  and a finite-to-one $\pi_*\mu$-a.e.~factor map $\pi:X\to Y$ such that $\pi_*\mu(Y_H)=1$.
\end{enumerate}
    Here the equivalence of (2) and (3) is immediate. Assume (3), then as in the proof of Theorem \ref{characterization-factor} there exists a factor map $\varphi: Y \to G/N_G(H)$, and the composition $\phi = \varphi \circ \pi: X \to G/N_G(H)$ is a factor map for $(X,G)$. Then the collection $\{\phi^{-1}(gN_G(H)): g \in G \}$ is a finite clopen partition of $X$ invariant under the action of $G$. The cylinder sets $U_{g,n}$ of the odometer action, see Section \ref{odom-action}, provide a basis for the topology on $X$, therefore, there exists $g \in G$ and $n \geq 1$ such that $U_{g,n} \subset \pi^{-1}(N_G(H))$. Then $g \Gamma_ng^{-1}$ stabilizes $U_{g,n}$, and thus $g \Gamma_n g^{-1} \subset N_G(H)$. Replacing $(\Gamma_n)_{n \in \mathbb N}$ by the chain of conjugate subgroups $(g\Gamma_ng^{-1})_{n \in \mathbb N}$ we obtain (1). Conversely, (1) implies the map $G/\Gamma_m\to G/N_G(H)$ given by $g\Gamma_m\to gN_G(H)$ is well defined, for every $m\geq n$, from which follows there exists a factor map $\phi:X\to G/N_G(H)$. Thus, using (2) of Theorem \ref{characterization-factor} we get (3). 
}
 \end{example}

In practice it may be difficult to find a group $G$ which has an almost normal finite subgroup. We describe one way to obtain such a group below in Example \ref{ex-farber}.

\begin{example}\label{ex-farber}
{\rm 
Let $\Gamma$ be a countable residually finite group. Consider $G=\Gamma\times A_p$, where $A_p$ is the alternating subgroup of $S_p$, the symmetric group on $p$ elements.   Let $K$ be a non-trivial proper subgroup of $A_p$. If $p\geq 5$,  then $A_p$ is simple, and so $1<[A_p:N_{A_p}(K)]=m<\infty$.
Let $H=\{1_{\Gamma}\}\times K$. We have $N_G(H)=\Gamma\times N_{A_p}(K)$ and $[G: N_G(H)]=m$. If $(\Gamma_n)_{n\in\NN}$ is a Farber sequence of subgroups of $\Gamma$, then $(\Gamma_n\times \{1_{A_p}\})_{n\in\NN}$ is a Farber sequence of $G$. The $G$-odometer $(X,G,\mu)$ associated to $(\Gamma_n\times \{1_{A_p}\})_{n\in\NN}$ satisfies condition (1) of Example \ref{Odometer}, therefore, there exist a minimal dynamical system $(Y,G)$ and a $|H|$-to-one factor map $\pi:X\to Y$ such that $G_y\in{\rm Conj}(H)$ for  $\pi^*\mu$-a.e.~$y\in Y$.
}
\end{example}

\section{From factors to essentially free  finite-to-one extensions}\label{sec-extensions}

In this section, we explore under which conditions minimal systems with an almost normal finite stabilizer are finite-to-one factors of topologically and essentially free systems.

Recall that the residual subgroup $R(G)$ of $G$ is the intersection of all the finite index subgroups of $G$. This is a normal subgroup of $G$ (see \cite{CC10}).

\begin{prop}\label{theo-extension1}
  Suppose that there exists an almost normal finite subgroup $H$ of $G$ such that $H\cap R(G)=\{1_G\}$. Let $(Y,G)$ be a dynamical system such that $Y_H\neq\emptyset$. Then there exist a dynamical system $(X,G)$, a factor map $\pi:X\to Y$, and a constant $m\geq 1$ such that
  $$
  \pi^{-1}Y_H\subseteq\{x\in X: G_x=\{1_G\}\},$$  
  and
  $$
 |\pi^{-1}\{y\}|\leq m, \mbox{ for all } y\in Y_H.
 $$
Moreover,
\begin{enumerate}
\item If $(Y,G)$ is minimal, then $(X,G)$ is minimal and topologically free.

\item If there exists an ergodic measure $\mu$ of $(Y,G)$ such that $\mu(Y_H)=1$, then $(X,G)$ is essentially free with respect to a lift of $\mu$. 
\end{enumerate}

   \end{prop}\label{residually-finite-case}

\begin{proof}
   By hypothesis, for every $h\in H\setminus\{1_G\}$ there exists a finite index normal subgroup $\Gamma_h$ of $G$ such that $h\notin \Gamma_h$. Since $H$ is finite, $\Gamma=\bigcap_{h\in H\setminus\{1_G\}}\Gamma_h$ is a normal finite index subgroup of $G$ verifying $H\cap \Gamma=\{1_G\}$. Observe that $gHg^{-1}\cap \Gamma=\{1_G\}$, for every $g\in G$. 
 
 From this, we deduce that the action of $G$ on $Y\times G/\Gamma$ given by $g(x,w\Gamma)=(gx,gw\Gamma)$ is free when restricted to $Y_H\times G/\Gamma$. Indeed, suppose that
  $$g(x,w\Gamma) = (gx,gw\Gamma) = (x,w\Gamma).$$
Then $gx = x$ implies that $g \in lHl^{-1}$, for some $l\in G$, and, since $\Gamma$ is normal, $g w\Gamma = w\Gamma$ implies that $g \in \Gamma$. Then we must have $g=1_G$. 
 Furthermore, the projection $\pi: Y\times G/\Gamma\to Y$  is a $[G:\Gamma]$-to-1 factor map on $Y\times G/\Gamma$. 

  Let $X\subseteq Y\times G/\Gamma$ be a minimal component of the action of $G$ on $Y\times G/\Gamma$. If $(Y,G)$ is minimal, then $\pi(X)=Y$, which implies that the restriction $\pi|_X$ of $\pi$ on $X$ is a factor map that satisfies $\pi_X^{-1}Y_H\neq \emptyset$, and since the points in $\pi_X^{-1}Y_H$ have trivial stabilizer, we deduce that $(X,G)$ has points with trivial stabilizer. This implies that $(X,G)$ is topologically free (see, for example, \cite[Lemma 2.1]{Joseph2024}).
 
 If $\mu$ is an ergodic measure of $(Y,G)$ such that $\mu(Y_H)=1$, then  $\mu\times \lambda$, where $\lambda$ is the normalized counting measure on $G/\Gamma$, is a lift of $\mu$ with respect to which $(Y\times G/\Gamma, G)$ is essentially free.
\end{proof}

For a subgroup $H$ of $G$ we denote by $Z_H$ the subgroup of all elements $g\in G$ such that $gh=hg$, for every $h\in H$. Denote by $C_H$ the \emph{center} of $H$, i.e.~$C_H=Z_H\cap H$. 
 
 \begin{remark}\label{quasi-general-case_0}
{\rm Every almost normal finite index subgroup $H$ of $G$ with trivial center satisfies the hypothesis of Theorem \ref{theo-extension1}. Indeed, the N/C Theorem (see, for example,  \cite{Scott1964}) ensures that $N_G(H)/Z_H$ is isomorphic to a subgroup of $\Aut(H)$, which implies that $[N_G(H):Z_H]$ is finite. On the other hand, since $H$ is almost normal, $[G:N_G(H)]$  is finite, which implies that $[G: Z_H]$ is finite. Then, if $C_H$ is the center of $H$,  we get $H\cap R(G)\subseteq H\cap Z_H=C_H=\{1_G\}$. }
\end{remark}

\begin{coro}\label{quasi-general-case}
Let $H$ be a finite subgroup of $G$ with a trivial center.  Let $(Y,G,\mu)$ be a system equipped with an ergodic measure $\mu$. Then the following statements are equivalent: 
\begin{enumerate}
 \item $\mu(Y_H)=1$.
 \item $\mu(Y_H)=1$, i.e.~$(Y,G,\mu)$ has no essential holonomy w.r.t.~$\mu$.
 \item There exist an essentially free system $(X,G)$ with respect to a lift of $\mu$, and a finite-to-one $\mu$-a.e.~factor map $\pi:X\to Y$.
\end{enumerate}
\end{coro} 
\begin{proof}
 If $\mu(Y_H)=1$,   Lemma \ref{finite-conjugacy-classes} implies that $H$ is almost normal, and so its stabilizer URS is ${\rm Conj}(H)$ and $(Y,G,\mu)$ has no essential holonomy w.r.t.~$\mu$, i.e.~(2) holds. By Remark \ref{quasi-general-case_0} and Proposition \ref{theo-extension1}, we get (3). The implication (3) to (1) follows from Proposition \ref{factor-topologically-free}.
\end{proof}

We now give a few conditions which imply that $R(G) \cap H = \{1_G\}$, for a group $G$ and an almost normal subgroup $H$.

The next result is a direct consequence of Propositions \ref{factor-topologically-free}, \ref{theo-extension1}, and Corollary \ref{finite-summarize-RF}.
\begin{coro}\label{coro-residuallyfinite}
 Let $G$ be a residually finite group. Let $(Y,G)$ be a dynamical system equipped with an ergodic measure $\mu$. Then the following statements are equivalent:
\begin{itemize}
 \item[(1)] There exists a finite subgroup $H$ of $G$ such that $\mu(Y_H)=1$.
\item[(2)]There exists a finite subgroup $H$ of $G$ such that $\mu(Y_H)=1$ and $(Y,G,\mu)$ has no essential holonomy w.r.t.~$\mu$.
 \item[(3)]There exist an essentially free system $(X,G)$ with respect to a lift of $\mu$, and a finite-to-one $\mu$-a.e.~factor map $\pi:X\to Y$.
 \item[(4)] There exists an almost normal finite subgroup $H$ of $G$ such that $Y_H\neq \emptyset$.

\end{itemize}
If $G$ is not allosteric and $(Y,G)$ is minimal,  then the previous sentences are equivalent to the following:
\begin{itemize}

\item[(5)] There exist an almost normal finite subgroup $H$ of $G$, a minimal topologically free dynamical system $(X,G)$, and a factor map $\pi:X\to Y$ such that for some $m\geq 1$, we have $|\pi^{-1}\{y\}|\leq m$, for every $y\in Y_H$.  
    
\end{itemize}
\end{coro}

 \begin{prop}\label{sufficient-condition}
 Let $(Y,G)$ be a minimal system such that $Y_H\neq\emptyset$, for some almost normal finite subgroup $H$ of $G$. Suppose there exist a minimal  system $(X,G)$, a factor map $\pi:X\to Y$, and a constant $m\geq 1$ such that
  $$
  \pi^{-1}(Y_H)\subseteq\{x\in X: G_x=\{1_G\}\},$$  
  and
  $$
 |\pi^{-1}\{y\}|\leq m, \mbox{ for all } y\in Y.
 $$
 Then for every $y\in \overline{Y_H(1_G)}$, we have $H\cap R(G_y) =\{1_G\}$.
 \end{prop}
 \begin{proof}
Let $y\in Y$, and let $\{g_1,\ldots,g_n\}$ be a set of representatives of the cosets of $G/N_G(H)$, with $g_1 = 1_G$. By Corollary \ref{converse-minimal}, there exists a unique $1\leq i \leq n$ such that $y\in \overline{Y_H(g_i)}$, and there is a factor map $\phi': Y \to G/N_G(H)$ such that $\phi'(y) = g_i N_G(H)$. Observe that if $x\in \pi^{-1}\{y\}$, then $G_x\subseteq g_iN_G(H){g_i}^{-1}$. Since $g_iHg_i^{-1}$ is a normal subgroup of $g_iN_G(H)g_i^{-1}$, we have $\langle g_iHg_i^{-1}\cup G_x\rangle=g_iHg_i^{-1}G_x\subseteq G_y$. Since $\pi$ is $m$-to-one,  from Lemma \ref{finite-stabilizer} we deduce that the group $g_iHg_i^{-1}G_x$ is of finite index in $G_y$. On the other hand, since $g_iHg_i^{-1}$ is finite, $G_x$ is of finite index in $g_iHg_i^{-1}G_x$, and then in $G_y$. From Lemma \ref{key}, we get that $G_x$ is a finite index subgroup of $G_y$ such that $g_iHg_i^{-1}\cap G_x=\{1_G\}$. This implies that $g_iHg_i^{-1}\cap R(G_y)=\{1_G\}$. Since $y \in Y$ is arbitrary, the proof is valid, in particular, for $y \in Y_H(g_1) = Y_H(1_G)$, and we obtain the statement of the proposition.
\end{proof}

\begin{remark}\label{final-remark}
{\rm In the context of Proposition \ref{sufficient-condition}, observe that $R(G_y)\subseteq R(G)\cap G_y$. Thus, $R(G)\cap H=\{1_G\}$ implies $R(G_y)\cap H=\{1_G\}$. These conditions are equivalent in the case $R(G_y)=R(G)\cap G_y$, that is, when every finite index subgroup of $G_y$ comes from a finite index subgroup of $G$. We do not know if in general, the conditions of Proposition \ref{sufficient-condition} are also sufficient.}    
\end{remark}
 
In order to describe our final example, we first prove two auxiliary lemmas, Lemma \ref{sufficient-condition-1} and Lemma \ref{lemma-trivial-core}. 

\begin{lemma}\label{sufficient-condition-1}
    Let $(Y,G)$ be a dynamical system, let $x \in Y$ be a point with a dense orbit. Let $H = G_x$, and suppose $H$ is an almost normal subgroup of $G$. If the union of conjugates $\cup_{g \in G} gHg^{-1}$ is not contained in the intersection of normalizers
      $$\bigcap_{g \in G} N_G(gHg^{-1}) = \bigcap_{g \in G} gN_G(H)g^{-1},$$
    then $R(G) \cap H = \{1_G\}$.  
\end{lemma}

\begin{proof}
    Without loss of generality we may assume that $H \subset R(G)$. Then for any finite index normal subgroup $N \subset G$ we have $H \subset N$. We will show that this cannot be the case if the hypothesis of the lemma is satisfied. 

    Consider the action of the group $G$ on the set of conjugates ${\rm Conj}(H)$ of $H$. Let $p$ be the cardinality of ${\rm Conj}(H)$. The action permutes the elements in ${\rm Conj}(H)$, and thus we obtain a homomorphism 
      $$\phi: G \to S_p,$$
    where $S_p$ is the symmetric group on $p$ elements.  The kernel $\ker \phi$ is a finite index normal subgroup of $G$, and so we must have $H \subset R(G) \subset \ker \phi$.

    Now suppose $h \in {\rm ker}(\phi)$, then its action on ${\rm Conj}(H)$ fixes all conjugates of $H$, i.e.~for all $g \in G$ we have 
      $$h(gHg^{-1})h^{-1} = gHg^{-1}.$$
    Conjugate both sides of the equation by $g^{-1}$, then
      $$(g^{-1}hg) H (g^{-1} h g)^{-1} = H,$$
  i.e. ~$g^{-1}h g \in N_G(H)$. Thus if there exists $h \in H$ and $g \in G$ such that $g^{-1} hg H gh^{-1} g^{-1} \ne H$, then $\phi(h)$ is a non-trivial permutation in $S_p$, which is not possible since $H \subset R(G)$. This proves the lemma.
\end{proof}

We finish the paper by describing another family of examples of actions with finite almost normal stabilizers. For that we first need the following technical lemma, whose proof we include for completeness.

\begin{lemma}\label{lemma-trivial-core}
     Let $(X,G)$ be a dynamical system, and suppose $x \in X$ has a dense orbit. Let $H = G_x$ be the stabilizer of the action of $G$ at $x$, and suppose the action $(X,G)$ is effective. Then $H$ has no non-trivial proper subgroups which are normal in $G$.
 \end{lemma}

\begin{proof}
Recall that $X_H$ denotes the set of points in $X$ with stabilizers conjugate to $H$. Since $x$ has a dense orbit, we have $\overline{X_H} = X$, and so by Lemma \ref{lemma_contention} for every $y \in X$ there is $g \in G$ such that $G_y$ contains the conjugate $gHg^{-1}$ as a subgroup. Now let $N \subset H$ be a subgroup of $H$ normal in $G$. Then for any $g \in G$ we have $gNg^{-1} = N \subset gHg^{-1}$, and it follows that for any $y \in X$ we have $N \subset G_y$. Since the action is effective, $N$ is trivial.
\end{proof}

\begin{example}\label{FC-central}
   {\rm  
   We say that a subgroup $G' < G$ is an \emph{FC-subgroup} of $G$ if every element in $G'$ has a finite number of conjugates in $G$ \cite{Neumann1951}. If such $G'$ is finitely generated, then it is straightforward to see that $G'$ is an almost normal subgroup of $G$. If $G'$ contains all elements of $G$ with finite conjugacy class, then $G'$ is called the \emph{FC-center} of $G$ \cite{FGK98}. A family of group actions with almost normal stabilizers can be constructed as the actions of FC-central extensions (see for instance \cite{EZ2020}) of finitely generated residually finite groups $G$. 

   Let $G$ be a finitely generated residually finite group, and let $\mathcal G: G = G_0 \supset G_1 \supset \cdots$ be an infinite decreasing sequence of finite index subgroups of $G$. Denote by $K(\mathcal G) = \bigcap_{i \geq 0} G_i$, and consider the $G$-odometer action $(X,G)$ on a Cantor set $X$ defined as in Section \ref{odom-action}. If the action of $(X,G)$ is effective, by Lemma \ref{lemma-trivial-core} the intersection $K(\mathcal G)$ has no non-trivial proper normal subgroups. For each $i \geq 0$, the intersection $C_i = \bigcap_{g \in G} gG_ig^{-1}$ is a normal subgroup of $G$ of finite index. Thus $G/C_i$ is a finite group, and we have an infinite decreasing chain of normal subgroups $G = C_0 \supset C_1 \supset \cdots$, see for instance \cite{DHL2016}. Then $\bigcap_{i \geq 1} C_i \subset K(\mathcal G)$, and so by Lemma \ref{lemma-trivial-core} $\bigcap_{i \geq 1} C_i = \{1_G\}$, which in particular implies that the $G$-odometer action $(X,G)$ can only be effective if $G$ is residually finite.

     Let $S$ be a generating set of $G$, and 
    consider the space of Cayley graphs of all finitely generated groups with $|S|$ generators with Gromov-Hausdorff metric. That is, for a group $(G,S)$ the Cayley graph $(K,S)$ has $G$ as the set of vertices, and there is an edge between $g_1,g_2 \in G$ marked by $s \in S$ if and only if $s g_1 = g_2$. The Cayley graph $(K,S)$ has the natural length structure, given by declaring each edge to have length $1$, and the induced metric $d_{K,S}$, such that $d_{K,S}(w_1,w_2)$ is the infimum of the lengths of connected paths joining $w_1$ and $w_2$. 
    Then the distance between two Cayley graphs $(K_\ell,S_\ell)$, $\ell =1,2$ is given by $d = 2^{-m}$, where $m>0$ is the maximal integer, such that the balls of radius $m$ around the identity in $(K_\ell,S_\ell)$ are identical graphs. The Gromov-Hausdorff metric on Cayley graphs induces the metric, and so a topology, on the space of $k$-marked groups $(G,S)$, where $S$ is the generating set of $G$ with $k$ elements. This topology is called the Cayley topology, or the Cayley-Grigorchuk topology. 

    Now, returning to our sequence of groups $\mathcal G$, let $S$ be a generating set for $G$, and let $(K_i,S)$ be the Cayley graph of the quotient group $G/C_i$, for $i \geq 0$. Then, since $\bigcap_{i \geq 0} C_i = \{1_G\}$, the sequence $(K_i,S)$ converges with respect to the Gromov-Hausdorff metric to the Cayley graph $K$ of $G$, and so the sequence $G/C_i$ converges to the group $G$ in the Cayley topology. An \emph{FC-central extension} is then constructed as follows, see \cite[Remark 9.2]{EZ2020} for proofs. Let $F_k$ be a group with a finite generating set of cardinality $k = |S|$, and suppose we have a sequence of quotient maps $q_i:F_k \to G/C_i$, and $q:F_k \to G$. For instance, one can take $F_k$ to be the free group with $k$ generators. Next, define the universal group of the sequence $(K_i,S)$ by
      $$C = F_k/ \bigcap_{i \geq 0} \ker q_i.$$
    Then there exists a homomorphism $\widetilde q: C \to G$, and every $w \in \ker \widetilde q$ has a finite conjugacy class in $C$ \cite[Remark 9.2]{EZ2020}.  See \cite[Remark 9.3]{EZ2020} for an example of a non-trivial FC-extension.
    Let $\Gamma$ denote a finitely generated subgroup of $C$ which maps surjectively onto $G$, such that $\Gamma \cap \ker \widetilde q$ is non-trivial and is not contained in the center of $\Gamma$.

    Now let $H < \ker \widetilde q \cap \Gamma$ be a finitely generated subgroup. It is straightforward to see that $H$ is almost normal, indeed, since every element in $H$ has a finite number of conjugates, then the intersection of the centralizers of the generators of $H$ is a subgroup of finite index in $\Gamma$, contained in the normalizer of $H$. Let $H' = \bigcap_{g \in \Gamma} gHg^{-1}$ be the core of $H$ in $\Gamma$, then $H'$ is normal in $\Gamma$. Since the property of having a finite number of conjugacy classes is preserved under quotients \cite{Neumann1951}, $H/H'$ is an almost normal subgroup of $\Gamma/H'$. Thus ${\rm Conj}(H/H')$ is an URS in $S(\Gamma/H')$ which also supports an atomic IRS. By \cite[Proposition 13]{AGV2014} there exists a measure-preserving action of $\Gamma/H'$ on a Borel probability space $(X,\mu)$ such that the set of conjugates of $H/H'$ is the support of the IRS associated to $(X,\mu)$. Alternatively, by \cite[Theorem 1.1 and Corollary 1.2]{MT20}, the set of conjugates of $H/H'$ in $\Gamma/H'$ can be realized as the stabilizer URS of a minimal action on a compact space. All these actions are examples of dynamical systems with almost normal stabilizers.
}
\end{example}

\bibliographystyle{abbrv}
\bibliography{lit}
 
\end{document}